\documentclass[11pt]{amsart}

\usepackage{tikz}
\usetikzlibrary{shapes.gates.logic.US,trees,positioning,arrows,cd,backgrounds}
\usepackage{pgfplots}
\pgfplotsset{compat=1.14}
\usepackage{algorithmicx}
\usepackage{algorithm}
\usepackage{algpseudocode}

\newcommand{\slant}[2]{{\raisebox{.2em}{$#1$}/\raisebox{-.2em}{$#2$}}}

\usepackage{amsmath,amssymb}
\allowdisplaybreaks[1]
\usepackage{url}
\usepackage{xfrac}
\usepackage{nccmath}

\usepackage{longtable}
\usepackage[shortlabels]{enumitem}
\usepackage{csquotes}
\usepackage{multirow}
\usepackage{multicol}
\usepackage{array,ragged2e}
\usepackage{chngcntr}
\counterwithin{figure}{section}

\usepackage[pdftex,
        colorlinks=false,
        linkcolor=blue,
        filecolor=blue,
        citecolor=red,
        draft=false,
        bookmarks,
        plainpages=false,
        bookmarksnumbered=true]{hyperref}

\newcommand{\al}{{\alpha}}
\newcommand{\Z}{{\mathbb{Z}}}
\newcommand{\C}{{\mathbb{C}}}

\newcommand{\N}{{\mathbb{N}}}
\renewcommand{\k}{{\mathbf{k}}}

\renewcommand{\sl}{{\mathfrak{sl}}}
\newcommand{\fg}{{\mathfrak{g}}}

\newcommand{\ft}{{\mathfrak{t}}}
\newcommand{\fu}{{\mathfrak{u}}}

\newcommand{\cD}{{\mathcal{D}}}

\newcommand{\cN}{{\mathcal{N}}}

\newcommand{\cO}{{\mathcal{O}}}

\newcommand{\Ad}{{\operatorname{Ad}}}
\newcommand{\Int}{{\operatorname{Int}}}
\newcommand{\ad}{{\operatorname{ad}}}
\newcommand{\Lie}{{\operatorname{Lie}}}

\newcommand{\GL}{{\operatorname{GL}}}
\newcommand{\SL}{{\operatorname{SL}}}
\newcommand{\Gad}{G_{\operatorname{ad}}}
\newcommand{\id}{{\operatorname{id}}}
\newcommand{\Hom}{{\operatorname{Hom}}}

\newcommand{\im}{{\operatorname{im}}}

\newcommand{\chara}{{\operatorname{char}}}
\renewcommand{\leq}{\leqslant}
\renewcommand{\geq}{\geqslant}

\newcommand{\btri}{\blacktriangle}
\newcommand{\wtri}{\vartriangle}
\newcommand{\ra}{\rightarrow}
\newcommand{\lra}{\longrightarrow}
\newcommand{\+}{\times}
\newcommand{\inv}[1]{\ensuremath{#1^{-1}}}

\newtheorem{thm}{Theorem}[section]
\newtheorem{cor}[thm]{Corollary}
\newtheorem{lem}[thm]{Lemma}
\newtheorem{prop}[thm]{Proposition}
\newtheorem{conj}[thm]{Conjecture}
\theoremstyle{definition}
\newtheorem{defn}[thm]{Definition}

\theoremstyle{remark}
\newtheorem{rem}[thm]{Remark}
\numberwithin{equation}{section}

\begin{document}

\title[nilpotent pieces in bad characteristic]{On the computation of the nilpotent pieces in bad characteristic for algebraic groups of type $G_2$, $F_4$, and $E_6$}
\author{Laura Voggesberger}
\address{Ruhr-Universität Bochum, Fakultät für Mathematik, Postfach 69, 44780 Bochum, Germany}
\email{laura.voggesberger@rub.de}
\begin{abstract}
Let $G$ be a connected reductive algebraic group over an algebraically closed field $\k$, and let $\Lie(G)$
be its associated Lie algebra. In his series of papers on unipotent elements in small characteristic, Lusztig defined a partition of the unipotent variety of
$G$. This partition is very useful when working with representations of $G$. Equivalently, one can
consider certain subsets of the nilpotent variety of $\fg$ called pieces. This approach appears in
Lusztig’s article from 2011.
The pieces for the exceptional groups of type $G_2, F_4, E_6, E_7$, and $E_8$ in bad characteristic have not
yet been determined. 
This article presents a solution, relying on computational techniques, to this problem for groups of type $G_2$, $F_4$, and $E_6$.
\end{abstract}
\maketitle
\section{Introduction}
Let $G$ be a connected reductive algebraic group over an algebraically closed field $\k$. There has been a lot of work on both the unipotent orbits of the conjugation action of $G$ on itself and the nilpotent orbits given by the adjoint action of $G$, when $G$ is simple, on its Lie algebra $\Lie(G)=\fg$. A full list of these orbits can be found, for example, in the book of Liebeck and Seitz, \cite{LiebeckSeitz}. One notes that the parametrisation of the orbits is different in certain characteristics: if $\chara(\k)=p$, where $p$ is a prime number, we say that $p$ is \textit{bad} for a simple group $G$ if $p=2$ and the root system of $G$ is not of type $A_n$, if $p=3$ and $G$ has type $G_2, F_4, E_6$ or $E_7$, and if $p=5$ and $G$ is of type $E_8$. In the case where $p$ is a bad prime, the number and structure of the nilpotent orbits may differ from the orbits found in other (\textit{good}) characteristic, for instance $\chara(k)=0$. In his series of papers on unipotent elements in small characteristic \cite{Lusztig1},\cite{Lusztig2},\cite{Lusztig3}, and \cite{Lusztig4}, Lusztig defines a partition of the unipotent variety and in \cite{Lusztig3} of the nilpotent variety into so-called \textit{pieces}. The pieces are parametrized by the orbits in good characteristic and were explicitly computed for $G$ of classical type, i.e. $A_n,B_n,C_n$ and $D_n$ in \cite{Lusztig3}.

Interestingly, there exist different definitions of partitions of the nilpotent variety. In \cite{Hesselink}, Hesselink defines a stratification of the nilpotent variety. Clarke and Premet define their own nilpotent pieces in \cite{ClarkePremet} and show that this leads to the same stratification as proposed by Hesselink. In \cite{Xue14}, Xue computes nilpotent pieces in $\fg^*$ for groups of type $F_4$ and $G_2$ using the definition for the nilpotent pieces in $\fg^*$ proposed by Clarke and Premet in \cite{ClarkePremet}. In \cite{Xue17}, Xue describes the Springer correspondence for the types $G_2$ and $F_4$ and uses it to compute nilpotent orbit representatives in $\fg$. We cannot find the nilpotent pieces from $\fg^*$ under this correspondence, as they are computed by using the definition by Clarke and Premet.

One would hope that the definitions of the nilpotent pieces as given by Lusztig and Clarke--Premet lead to the same object, and this is indeed the case for algebraic groups of classical type. We expect the equality to hold in the exceptional cases as well.

Additionally, we do not yet know if the nilpotent pieces (defined by Lusztig) form a partition of the nilpotent variety in these cases. In this paper we are able to prove with the help of computations done in {\sc Magma},\cite{Magma}, that this is the case for $p=2,3$ in $G$ of type $G_2,F_4$ and $E_6$, resulting in the following theorem.

In this paper we work with the definition of the nilpotent pieces as proposed by Lusztig unless otherwise specified.

\begin{thm}[Nilpotent pieces in $G_2,F_4$, and $E_6$]\label{main_theorem}
Let $G$ be a simple algebraic group of type $G_2$, $F_4$ or $E_6$ over an algebraically closed field $\k$ with $\chara(\k)=2$ or $3$.
Let $\tilde{G}$ be of the same type as $G$ in good characteristic.
We can choose a list of representatives $x_i$, $i=1,\ldots,m$ for the nilpotent orbits in the Lie algerba $\fg$ of $G$ such that
\begin{enumerate}
	\item $x_i=\sum_{\al\in\Phi^+}\lambda_{i,\al}e_{\al},\quad \lambda_{i,\al}\in\{0,1\},\quad \langle e_{\al}\rangle=\fg_{\al}$, and
	\item $\tilde{x}_i=\sum_{\al\in\Phi^+}\tilde{\lambda}_{i,\al}\tilde{e}_{\al}$ in good characteristic such that\\
	 $\tilde{\lambda}_{i,\al}=0\Leftrightarrow\lambda_{i,\al}=0$ and $\tilde{\lambda}_{i,\al}=1\Leftrightarrow\lambda_{i,\al}=1$.
\end{enumerate}
Furthermore, let $\tilde{\cO}_{\delta}$ be the nilpotent orbit in good characteristic, described by a weighted Dynkin diagram $\delta$.\\
Then the nilpotent piece with respect to the weighted Dynkin diagram $\delta$ and the group $G$ is given by
\begin{align*}
\cN_{\fg}^{\btri_{\delta}}=\bigsqcup_{\tilde{x}_i\in\tilde{\cO}_{\delta}}\cO_{x_i}.
\end{align*}
In particular, the nilpotent pieces $\cN_{\fg}^{\btri_{\delta}}$ form a partition of $\cN_{\fg}$ and are in bijection with the nilpotent orbits in good characteristic.
\end{thm}
In order to understand and prove Theorem \ref{main_theorem}, we will give a short introduction to the definition of the nilpotent pieces in section \ref{introduction}. This will be followed by some auxiliary results on pieces and the sets associated to them. In section \ref{computation} we will finally introduce a computational approach where we describe how to explicitly compute the nilpotent orbits contained in each piece before stating our results in section \ref{results}. While a similar statement has yet to be proved for groups of type $E_7$ and $E_8$, we expect similar results to hold:
\begin{conj}\label{conjectureA}
The results stated in Theorem \ref{main_theorem} should hold for simple algebraic groups of type $E_7$ and $E_8$ in bad characteristic.
\end{conj}

 We hope that these results can provide a deeper understanding of the nilpotent variety and help to study representations of algebraic groups. For instance, knowledge of the nilpotent pieces should help to work out specifics for generalised Gelfand--Graev representations in small characteristics, see \cite{Geck20}. In this paper, Geck describes a way to define generalised Gelfand--Graev representations in small characteristic. This relies heavily on a linear map $\lambda$ being in ``sufficiently general position'' which can be checked with the knowledge of the nilpotent pieces.
\section{Preliminaries}\label{introduction}
Let $\k$ be an algebraically closed field with char$(\k)=p\geq 0$.
Let $G$ be a connected reductive algebraic group over $\k$ with Lie algebra $\fg$. We fix a maximal torus $T\subseteq G$ and denote by $\Phi\subseteq X(T)$ the corresponding root system consisting of characters of $T$. For the basics see also \cite[Chapters 3 and 6-9]{MalleTesterman}.
\subsection{Root systems and algebraic groups}
 Let $\Pi\subseteq \Phi$ be the simple roots with respect to a Borel subgroup $T\subseteq B$ of $G$. Every root in $\Phi$ can be written either as a non-negative or non-positive integral linear combination of the simple roots in $\Pi$. Therefore we can define
\begin{equation*}
\medmath{
\Phi^+:=\left.\left\{\alpha=\sum_{\beta\in\Pi} c_{\beta}\beta\in\Phi~\right|~ c_{\beta}\in\Z_{\geq 0}\right\},\quad
\Phi^-:=\left.\left\{\alpha=\sum_{\beta\in\Pi} c_{\beta}\beta\in\Phi~\right|~c_{\beta}\in\Z_{\leq 0}\right\}.
}
\end{equation*}

The algebraic group $G$ acts on its Lie algebra $\fg$ via the adjoint map $\Ad: G\ra\GL(\fg)$.
The root system $\Phi$ of $G$ can be characterised by its Dynkin diagram.

\begin{center}
\begin{figure}[h]
\begin{tikzpicture}[scale=.3]
	\draw (-0.5,0) node {$A_n$};
    \draw[thick] (1 cm,0) circle (.2cm);
    \draw[thick] (2.5 cm,0) circle (.2cm);
    \draw[thick] (5 cm,0) circle (.2cm);
    \draw[thick] (6.5 cm,0) circle (.2cm);
    
    \draw[thick] (1.2 cm,0) -- +(1.1 cm,0);
    \draw[thick,dotted] (2.7 cm,0)--+(2.1cm,0);
    \draw[thick] (5.2 cm,0) -- +(1.1 cm,0);
    \draw (10.5,0) node {$B_n$};
    \draw[thick] (12 cm,0) circle (.2cm);
    \draw[thick] (13.5 cm,0) circle (.2cm);
    \draw[thick] (16 cm,0) circle (.2cm);
    \draw[thick] (17.5 cm,0) circle (.2cm);
    
    \draw[thick] (12.2 cm,0) -- +(1.1 cm,0);
    \draw[thick,dotted] (13.7 cm,0)--+(2.1cm,0);
    \draw[thick] (16.2 cm,0.1) -- +(1.1 cm,0);
    \draw[thick] (16.2 cm,-0.1) -- +(1.1 cm,0);
    \draw[thick] (16.6cm, 0.2cm) --(16.95 cm,0)--(16.6cm, -0.2cm);
    \draw (21.5,0) node {$C_n$};
    \draw[thick] (23 cm,0) circle (.2cm);
    \draw[thick] (24.5 cm,0) circle (.2cm);
    \draw[thick] (27 cm,0) circle (.2cm);
    \draw[thick] (28.5 cm,0) circle (.2cm);
    
    \draw[thick] (23.2 cm,0) -- +(1.1 cm,0);
    \draw[thick,dotted] (24.7 cm,0)--+(2.1cm,0);
    \draw[thick] (27.2 cm,0.1) -- +(1.1 cm,0);
    \draw[thick] (27.2 cm,-0.1) -- +(1.1 cm,0);
    \draw[thick] (27.95 cm,0.2)--(27.6cm, 0cm) --(27.95 cm,-0.2);
    \draw (-1.5,-4) node {$D_n$};
    \draw[thick] (0 cm,-4) circle (.2cm);
    \draw[thick] (1.5 cm,-4) circle (.2cm);
    \draw[thick] (4 cm,-4) circle (.2cm);
    \draw[thick] (5.5 cm,-2.5) circle (.2cm);
    \draw[thick] (5.5 cm,-5.5) circle (.2cm);
    
    \draw[thick] (0.2 cm,-4) -- +(1.1 cm,0);
    \draw[thick,dotted] (1.7 cm,-4)--+(2.1cm,0);
    \draw[thick] (5.4cm, -2.6)--(4.1 cm,-3.8);
    \draw[thick](4.1 cm,-4.2)-- (5.4cm, -5.4);
    \draw (8.5,-4) node {$G_2$};
    \draw[thick] (10 cm,-4) circle (.2cm);
    \draw[thick] (11.5 cm,-4) circle (.2cm);

    \draw[thick] (10.2 cm,-4) -- +(1.1 cm,0);
    \draw[thick] (10 cm,-3.8) -- +(1.5 cm,0);
    \draw[thick] (10 cm,-4.2) -- +(1.5 cm,0);
    \draw[thick] (10.85cm, -3.8cm) --(10.5 cm,-4)--(10.85cm, -4.2cm);
    \draw (14.5,-4) node {$F_4$};
    \draw[thick] (16 cm,-4) circle (.2cm);
    \draw[thick] (17.5 cm,-4) circle (.2cm);
    \draw[thick] (19 cm,-4) circle (.2cm);
    \draw[thick] (20.5 cm,-4) circle (.2cm);

    \draw[thick] (16.2 cm,-4) -- +(1.1 cm,0);
    \draw[thick] (17.7 cm,-3.9) -- +(1.1 cm,0);
    \draw[thick] (17.7 cm,-4.1) -- +(1.1 cm,0);
    \draw[thick] (18.05cm, -3.8cm) --(18.4 cm,-4)--(18.05cm, -4.2cm);
    \draw[thick] (19.2 cm,-4) -- +(1.1 cm,0);
    \draw (23.5,-4) node {$E_6$};
    \draw[thick] (25 cm,-4) circle (.2cm);
    \draw[thick] (26.5 cm,-4) circle (.2cm);
    \draw[thick] (28 cm,-4) circle (.2cm);
    \draw[thick] (29.5 cm,-4) circle (.2cm);
    \draw[thick] (31 cm,-4) circle (.2cm);
    \draw[thick] (28 cm,-5.5) circle (.2cm);
    
    \draw[thick] (25.2 cm,-4) -- +(1.1 cm,0);
    \draw[thick] (26.7 cm,-4) -- +(1.1 cm,0);
    \draw[thick] (28.2 cm,-4) -- +(1.1 cm,0);
    \draw[thick] (29.7 cm,-4) -- +(1.1 cm,0);
    \draw[thick] (28 cm,-4.2) -- +(0,-1.1cm);
    \draw (2.5,-8) node {$E_7$};
    \draw[thick] (4 cm,-8) circle (.2cm);
    \draw[thick] (5.5 cm,-8) circle (.2cm);
    \draw[thick] (7 cm,-8) circle (.2cm);
    \draw[thick] (8.5 cm,-8) circle (.2cm);
    \draw[thick] (10 cm,-8) circle (.2cm);
    \draw[thick] (11.5 cm,-8) circle (.2cm);
    \draw[thick] (7 cm,-9.5) circle (.2cm);
    
    \draw[thick] (4.2 cm,-8) -- +(1.1 cm,0);
    \draw[thick] (5.7 cm,-8) -- +(1.1 cm,0);
    \draw[thick] (7.2 cm,-8) -- +(1.1 cm,0);
    \draw[thick] (8.7 cm,-8) -- +(1.1 cm,0);
    \draw[thick] (10.2 cm,-8) -- +(1.1 cm,0);
    \draw[thick] (7 cm,-8.2) -- +(0,-1.1cm);
    \draw (15.5,-8) node {$E_8$};
    \draw[thick] (17 cm,-8) circle (.2cm);
    \draw[thick] (18.5 cm,-8) circle (.2cm);
    \draw[thick] (20 cm,-8) circle (.2cm);
    \draw[thick] (21.5 cm,-8) circle (.2cm);
    \draw[thick] (23 cm,-8) circle (.2cm);
    \draw[thick] (24.5 cm,-8) circle (.2cm);
    \draw[thick] (26 cm,-8) circle (.2cm);
    \draw[thick] (20 cm,-9.5) circle (.2cm);
    
    \draw[thick] (17.2 cm,-8) -- +(1.1 cm,0);
    \draw[thick] (18.7 cm,-8) -- +(1.1 cm,0);
    \draw[thick] (20.2 cm,-8) -- +(1.1 cm,0);
    \draw[thick] (21.7 cm,-8) -- +(1.1 cm,0);
    \draw[thick] (23.2 cm,-8) -- +(1.1 cm,0);
    \draw[thick] (24.7 cm,-8) -- +(1.1 cm,0);
    \draw[thick] (20 cm,-8.2) -- +(0,-1.1cm);
\end{tikzpicture}
\small
 \label{fig:dynkindiagrams}\caption{The Dynkin diagrams of indecomposable root systems}
 \normalsize
\end{figure}
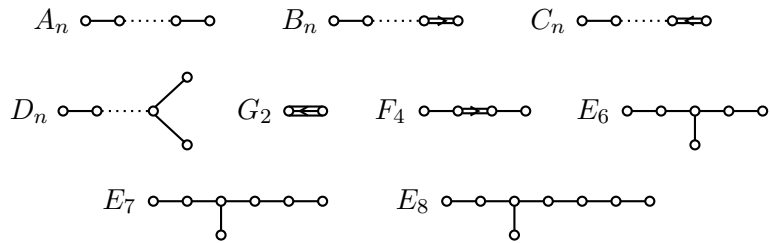
\end{center}
\vspace{-0.8cm}
\begin{rem}
All possibilities for Dynkin diagrams of indecomposable root systems are displayed in figure \ref{fig:dynkindiagrams} above, see for example \cite[Theorem 11.4]{Humphreys80}.
\end{rem}
For each root $\alpha\in\Phi$ there exists an isomorphism of algebraic groups onto a closed subgroup $U_{\al}$ of $G$
\begin{align}\label{ualpha}
u_{\alpha}:(\k,+)\lra U_{\alpha}\subseteq G.
\end{align}
These maps are unique up to multiplication by constants in $\k^{\+}$, \cite[Theorem 8.17]{MalleTesterman}.
The images of the maps $u_{\alpha}$ are called the \textbf{root subgroups} $U_{\alpha}$. This also means that every element $u\in U_{\alpha}$ can be written as $u=u_{\al}(c_{\al})$ for some uniquely determined $c_{\alpha}\in\k$.\\
In the Lie algebra $\fg$ we have the corresponding one-dimensional subspaces $\fg_{\alpha}:=\Lie(U_{\al}$. Then $\fg:=\bigoplus_{\al\in\Phi}\fg_{\al}\oplus\Lie(T)$. We fix a Chevalley basis $$\{e_{\al}\mid\al\in\Phi\}\cup\{h_i\mid i\in\{1,\ldots,|\Pi|\}$$ where $\{h_i\mid i\in\{1,\ldots,|\Pi|\}$ forms a basis of $\Lie(T)$ and the elements $e_{\alpha}$ generate the spaces $\fg_{\alpha}$ for all $\al\in\Phi$.\\
The \textbf{Weyl group} $W$ of $G$ is defined by $W:=\slant{N_G(T)}{T}$.\\
Furthermore, the \textbf{Bruhat decomposition} of elements of $G$, is defined as follows, see for example \cite[Theorem 11.17]{MalleTesterman} or \cite[28.3 and 28.4]{Humphreys75}.
\begin{defn}[Bruhat decomposition]\label{bruhatdecomp}
Use the same notation as before.
The \textbf{Bruhat decomposition} of an element $g\in G$ is given by $g=u'tn_wu$ where
\begin{enumerate}[itemsep=8pt]
\item $n_w\in N_G(T)$ is a representative of $w\in W$ in $N_G(T)$. Note that in order to ensure uniqueness, the $n_w$ are fixed once chosen.
\item $u\in\prod_{\alpha\in\Phi^+}U_{\alpha}=:U$, $U$ is a subgroup of $G$
\item $t\in T$, and
\item $u'\in\prod_{\substack{\alpha\in\Phi^+\\w.\alpha\in\Phi^-}}U_{\alpha}=:U^-_w$, $U^-_w$ is a subgroup of $G$.
\end{enumerate}
This decomposition is uniquely defined. In particular, we have $$G=\bigcup_{w\in W} Bn_wB,$$
where $B$ is a Borel subgroup such that $T\subseteq B$ and $\Phi^+$ are the positive roots with respect to $B$.
\end{defn}
The Weyl group $W$ of $G$ is generated by the simple reflections $s_{\alpha}$ for $\alpha\in\Pi$. The representative of $s_{\alpha}$ in $G$, denoted by $n_{\alpha}$, can be chosen as $n_{\alpha}=u_{\al}(1)u_{-\al}(-1)u_{\al}(1)$, see \cite[Section 8.4 and Theorem 8.17]{MalleTesterman} or \cite[1.9, p.19]{Carter85}.
\subsection{The set \texorpdfstring{$\cD_G$}{}}\label{cDG}
Consider a homomorphism of algebraic groups $\delta:\k^{\+}\lra G$ mapping any element $c\in \k^{\+}$ to an element $t_c\in T\subseteq G$. Then $\delta(\k^{\+})\subseteq T$ and we can apply the roots in $\Phi$ to $\im(\delta)$.\\
For $\alpha\in \Phi$, the map $\alpha\circ\delta:\k^{\+}\lra \k^{\+}$ is a homomorphism of algebraic groups and therefore we have $(\alpha\circ\delta)(c)=c^n$ for all $c\in \k^{\+}$ and some $n\in\Z$. Let $\langle \alpha,\delta\rangle:=n$. This defines a bilinear form. By abuse of notation we can define a linear map
\begin{align}\label{eqdelta}
\delta:\Phi\lra\Z,\quad \alpha\longmapsto\langle \alpha,\delta\rangle.
\end{align}
In the following text we are interested in a particular subset of these maps $\delta$ as above. We follow the construction of this subset as given in \cite[1.1]{Lusztig3}.\\
Let $G'$ be a connected reductive algebraic group of the same type as $G$, that is $G'$ has the same root system as $G$, but defined over $\C$. We recall that $T\subseteq G$ is a maximal torus in $G$ and let $Y_G:=\Hom(\k^{\+},T)$ be its cocharacter group. Similarly, let $T'\subseteq G'$ be a maximal torus and $Y_{G'}:=\Hom(\C^{\+},T')$. Both $N_G(T)$ and $N_{G'}(T')$ act on $Y_G$ and $Y_{G'}$ by conjugation respectively. Let $W$ be the Weyl group of $G$. As $G'$ is of the same type as $G$, $W$ is (up to isomorphism) also the Weyl group of $G'$. Now $W=\slant{N_G(T)}{T}$ acts on $T$ by conjugation and therefore $W$ also acts on $Y_G$ and by the same argument on $Y_{G'}$. As all maximal tori are conjugate in $G$ and $\im(\delta)\subseteq \tilde{T}$ for any map $\delta\in \Hom(\k^{\+},G)$ and a maximal torus $\tilde{T}$, we find a bijection between the set of orbits $G\backslash\Hom(\k^{\+},G)$ and $W\backslash Y_G$. As $W\backslash Y_G=W\backslash Y_{G'}$ (up to isomorphism), we can find a bijection between $G\backslash\Hom(\k^{\+},G)$ and $G'\backslash\Hom(\C^{\+},G')$.\\
Let
$$\cD_{G'}:=\left\{
\begin{array}{l|l}
\multirow{2}*{$f\in\Hom(\C^{\+},G')$} & \text{there exists } h\in\Hom(\SL_2(\C),G')\text{ s.t. }\\
 & h\begin{pmatrix}
a& 0\\0&\inv{a}
\end{pmatrix}= f(a)\text{ for all }a\in \C^{\+}
\end{array}
\right\}.$$
For the group $G$, we define the set $\cD_G\subseteq\Hom(\k^{\+},G)$ as follows. Let $\delta\in \Hom(\k^{\+},G)$. Then the element $\delta$ is contained in $\cD_G$ if and only if there exists an element $\delta'\in\cD_{G'}$ which corresponds to $\delta$ under the bijection of the orbits in $G'\backslash\Hom(\C^{\+},G')$ and  $G\backslash\Hom(\k^{\+},G)$.
\subsection{Weighted Dynkin diagrams}
In the following paragraph let $G$ be defined over an algebraically closed field $\k$, such that the characteristic of $\k$ is good for $G$. In this section, we follow the work of Carter \cite{Carter85}. Note that the results there only hold for the characteristic $p$ large enough. Nevertheless, the results are still true in general for good characteristic, see for example By Pommerening \cite{Pommerening1} and \cite{Pommerening2}, as well as Premet \cite{Premet}.

Let $0\neq e\in\fg=\Lie(G)$ be a nilpotent element. We can embed $e$ in a 3-dimensional subalgebra $\langle e,h,f \rangle$ of $\fg$ isomorphic to $\sl_2(\k)$ (\cite[Theorem 5.3.2]{Carter85}). This subalgebra determines a map $\gamma:\k^{\+}\ra G$ as follows: By \cite[Theorem 5.4.8]{Carter85} $\fg$ is a direct sum of irreducible $\sl_2(\k)$-modules with basis $x_1,\ldots,x_j$ and representation $\rho_j$ of $\sl_2(\k)$ such that
\begin{align*}
\hspace{1.5cm}ex_i&=x_{i+1}, &i&=1,2,\ldots,j-1, \hspace{1cm} &ex_j&=0,\hspace{1.5cm}\\
hx_i&=(2i-j-1)x_i,  &i&=1,2,\ldots,j,\tag{\theequation}\label{repsl2}\\
fx_{i+1}&=i(j-i)x_i, &i&=1,2,\ldots,j-1,\, &fx_j&=0,
\end{align*} as defined in \cite[Section 5.4]{Carter85}. Let $c\in\k$ and $x_1,\ldots,x_j$ be the basis of an irreducible $\sl_2(\k)$-module in $\fg$ as above. Then $\gamma$ acts on this basis by
\begin{align}\label{gammasl2}
\gamma(c).x_i=c^{2i-j-1}x_i,
\end{align} 
and $\gamma(c)$ describes an action of $\begin{pmatrix}
c & 0\\0&\inv{c}
\end{pmatrix}\in\SL_2(\k)$ on $\fg$, see \cite[Proposition 5.5.6]{Carter85}.\\
One can choose a maximal torus $T\subseteq G$ such that $\im(\gamma)\subseteq T$ and $\Pi\subseteq\Phi$ such that $\gamma(\Pi)\subseteq \{0,1,2\}$ (cf \cite[5.6.7]{Carter85}).
\begin{rem}\label{remgamma}~
For each map $\delta\in\cD_G$ one can find a map $\gamma$ as defined above and vice versa.\\
If $e=0$ we define the trivial map $\delta(c)=1$ for all elements $c\in\k^{\+}$.
\end{rem}
\begin{defn}[Weighted Dynkin diagrams]
Let $\gamma$ be as above. Then we can define the linear map $\gamma:\Phi\ra\Z$, where $\gamma(\alpha)=\langle \gamma,\alpha\rangle$, as in \eqref{eqdelta}.\\
As $\gamma$ is a linear map, it is determined by its values on the set of simple roots. This means that, instead of giving $\gamma$, we can take the Dynkin diagram corresponding to $\Phi$ and assign to the node for the root $\alpha\in\Pi$ the value $\gamma(\alpha)$. By Remark \ref{remgamma} we can find a system of simple roots such that the nodes are labelled by 0,1 or 2.
The resulting diagram is called the \textbf{weighted Dynkin diagram} of $\gamma$.\\
Let $\gamma$ be as in \eqref{gammasl2},and let $\delta$ correspond to $\gamma$ under the bijection in section \ref{cDG}. We will call $\delta$ the map \textbf{arising} from the weighted Dynkin diagram of $\gamma$ if $\delta(\alpha)=\gamma(\alpha)$ for all $\alpha\in\Pi$ (and hence all $\al\in\Phi$).
\end{defn}
\begin{rem}
Even though the weighted Dynkin diagrams arise from the above construction in good characteristic, we can define corresponding maps $\eta_{\delta}:\Phi\lra\Z$ such that $\eta_{\delta}(\alpha)$ corresponds to the weight of the node belonging to the simple root $\alpha\in\Pi$ in every characteristic. If $\delta\in\cD_G$ and $\eta_{\delta}:\Phi\ra\Phi, \al\mapsto\langle\delta,\al\rangle$ is the corresponding linear map on $\Phi$, we will also write $\delta$ instead of $\eta_{\delta}$ by abuse of notation.
\end{rem}
\subsection{The sets \texorpdfstring{$\fg_i^{\delta},~\fg_{\geq i}^{\delta}$}{}, and \texorpdfstring{$\fg_2^{\delta!}$}{}} 
Following \cite[Section 1]{Lusztig3}, we will define certain subsets of the Lie algebra $\fg$ which will eventually lead us to the definition of the nilpotent pieces.\\
From now on, let the characteristic of $\k$ be arbitrary. Let  $\delta\in \cD_G$ be a cocharacter, that is $\delta(\k^{\+})\subseteq T$, and $i\in \Z$. We can define subspaces of the Lie algebra $\fg$ depending on the weighted Dynkin diagram corresponding to $\delta$. These subspaces are crucial in the definition of the \textit{nilpotent pieces} whose union will -- in good characteristic -- prove to be the nilpotent variety.\\
As $G$ is a connected reductive group we have $\fg=\ft\oplus\bigoplus_{\alpha\in\Phi}\fg_{\alpha}$ where $\ft=$Lie$(T)$ and the $\fg_{\alpha}=\{x\in\fg\mid\Ad(t)(x)=\alpha(t)x~\text{ for all }t\in T\}=\Lie(U_{\al})$ are the one-dimensional rootspaces, see \cite[Theorem 8.17]{MalleTesterman}.\\
Define $$\fg_i^{\delta}:=\{x\in\fg\mid\Ad(\delta(a))(x)=a^ix~\text{ for all }a\in \k^{\+}\}.$$
Clearly, we have $\fg_i^{\delta}=\bigoplus_{\delta(\alpha)=i}\fg_{\alpha}$ for all $i\in\Z\setminus\{0\}$. If $i=0$ we have $\fg_0^{\delta}=\ft\oplus\bigoplus_{\delta(\alpha)=0}\fg_{\alpha}$. Note that for $i\neq 0$ this is not a Lie algebra: We have $[\fg_{\al},\fg_{\beta}]\subseteq \fg_{\al+\beta}$ for roots $\al,\beta\in\Phi$ such that $\al+\beta\in\Phi$. In particular, $\delta(\alpha+\beta)=2i$ if $\fg_{\al},\fg_{\beta}\subseteq\fg_i^{\delta}$, see also \cite[Proposition 5.5.7]{Carter85}.\\
Similarly, for $i\in\Z$ define the sets $$\fg_{\geq i}^{\delta}:=\bigoplus_{j\geq i}\fg_j^{\delta}.$$
By the same argument as above, we can see that for $i\geq 0$ the $\fg_{\geq i}^{\delta}$ are in fact Lie subalgebras.\\
We can define the corresponding subgroups $G_{\geq i}^{\delta}$, $i\geq 0$, such that Lie($G_{\geq i}^{\delta})=\fg_{\geq i}^{\delta}$ by
\begin{align*}
G_{\geq i}^{\delta}=\langle U_{\alpha}\mid \ \alpha\in\Phi, \langle \alpha,\delta\rangle\geq i\rangle\quad\text{if } i\neq 0,
\end{align*}
and 
\begin{align*}
G_{\geq 0}^{\delta}=\langle T,U_{\alpha}\mid \ \alpha\in\Phi, \langle \alpha,\delta\rangle\geq 0\rangle.
\end{align*}
In particular, $G_{\geq 0}^{\delta}$ is a parabolic subgroup of $G$ with the Levi subgroup $G_{0}^{\delta}=\langle T,U_{\alpha}\mid \ \alpha\in\Phi, \langle \alpha,\delta\rangle= 0\rangle$ (\cite[1.2]{Lusztig3}).\\
We have an equivalence relation on the set $\cD_G$ where $\delta\sim\delta'$ if $\fg_{\geq i}^{\delta}=\fg_{\geq i}^{\delta'}$ for all $i\in\Z$. We denote by 
$$\wtri_{\delta}:=\{\delta'\in\cD_G\mid \fg_{\geq i}^{\delta}=\fg_{\geq i}^{\delta'}\text{ for all } i\in\Z\}$$
 the equivalence class of a map $\delta\in\cD_G$.
\begin{defn}[$\fg_2^{\delta!}$]
Let $x\in\fg$. Then $G_x=\{g\in G\mid \Ad(g)x=x\}$ is the stabilizer of $x$ in $G$.\\
Now we can define 
$$\fg_2^{\delta!}:=\{x\in\fg_2^{\delta}\mid G_x\subseteq G_{\geq 0}^{\delta}\}.$$
Note that in general, $0\notin \fg_2^{\delta!}$, so $\fg_2^{\delta!}$ is not a subspace of $\fg$.
\end{defn}
\section{The Nilpotent Pieces}\label{thenilpieces}
From now on write $\fg_i^{\wtri},~ \fg_{\geq i}^{\wtri}$ instead of $\fg_i^{\delta},~ \fg_{\geq i}^{\delta}$ where $\wtri=\wtri_{\delta}$.\\
We have the obvious isomorphism of vector spaces
\begin{align*}
\fg^{\wtri}_2\xrightarrow{~\sim~} \slant{\fg_{\geq 2}^{\wtri}}{\fg_{\geq 3}^{\wtri}}.
\end{align*}
Let $\Sigma^{\wtri}$ be the image of $\fg_2^{\delta!}$ under this isomorphism. Furthermore, using the natural map
\begin{align*}
\pi:\fg_{\geq 2}^{\wtri}\lra \slant{\fg_{\geq 2}^{\wtri}}{\fg_{\geq 3}^{\wtri}}
\end{align*} define $\sigma^{\wtri}:=\pi^{-1}(\Sigma^{\wtri})$.
We let $\blacktriangle_{\delta}$  be the $G$-orbit of $\wtri_{\delta}$ via the conjugation action of $G$. We will sometimes also refer to this orbit by $\btri$.
The following definition is due to Lusztig, \cite[A.6.]{Lusztig3}.
\begin{defn}[Nilpotent Pieces]\label{nilpotentpieces}
The sets
$$\cN_{\fg}^{\btri_{\delta}}=\bigcup_{\wtri\in\btri_{\delta}}\sigma^{\wtri}$$
where $\btri_{\delta}$ runs over all $G$-orbits represented by the set $\{\wtri_{\delta}\mid \delta\in\cD_G\}$, are the \textbf{nilpotent pieces} in $\fg=$Lie$(G)$.
\end{defn}
\begin{thm}[see {\cite[A.6]{Lusztig3}}]
The pieces $\cN_{\fg}^{\blacktriangle}$ form a partition of the nilpotent variety $\cN_{\fg}$ of $\fg$ if $G$ is simple of classical type $A,B,C$ or $D$ in any characteristic.
\end{thm}
\begin{rem}[The pieces in good characteristic]
In good characteristic the nilpotent orbits in $\fg$ correspond bijectively to the weighted Dynkin diagrams, see \cite[5.6, 5.11]{Carter85}.\\
Let $\k$ be of good characteristic for $G$ and let the element $e\in\fg$ be in the nilpotent orbit corresponding to the weighted Dynkin diagram arising from a map $\delta\in\wtri_{\delta}\in\btri$. 
By \cite[5.5.7]{Carter85} we have $e\in\fg_2^{\delta}$ and by \cite[5.6.2]{Carter85} the orbit of $e$ under the action of $C_G(\delta(\k^{\+}))=G_0^{\delta}$ is a dense open subset of $\fg_2^{\delta}$. Let $\cO_e$ be this orbit. By \cite[1.2.(a)]{Lusztig3} it follows that $\cO_e=\fg_2^{\delta!}$. This means that in good characteristic each piece $\cN_{\fg}^{\btri}$ is given by the nilpotent orbit in $G$ of an element corresponding to the weighted Dynkin diagram arising from $\delta\in\wtri_{\delta}\in\btri$. Therefore the above theorem continues to hold for exceptional type in good characteristic.
It is not yet known whether this is true for $G$ of exceptional type in bad characteristic.
\end{rem}
The pieces are explicitly known for groups of type $A,B,C$ and $D$ in all characteristics (see \cite{Lusztig3}). However, we do not yet know the pieces in bad characteristic for simple groups of exceptional type, i.e. $G=G_2,F_4, E_6$ or $E_7$ for $p=2,3$ and $G=E_8$ for $p=2,3$ or $5$.\\
For further computations we note that it is enough to assume that $G$ is a semisimple adjoint group.
\begin{prop}
Let $G$ be a connected reductive algebraic group. Then the pieces of $G$ are the same as those of the semisimple adjoint group of the same type.
\end{prop}
\begin{proof}
Let $\Gad$ be an adjoint group of the same type as $G$. We first show that the pieces in $\fg_{\ad}:=\Lie(\Gad)$ are the same as in $\fg$.\\
There exists a central isogeny $\pi: G\ra \Gad$ by \cite[24.1]{Borel} with differential $d\pi:\fg\ra\fg_{\ad}$. By \cite[Proposition 22.4]{Borel} $d\pi$ is injective on each $\fu_{\al}:=\Lie(U_{\al})$ and by \cite[Corollary 22.5]{Borel} $\fg_{\ad}=d\pi(\fg)+\ft'$ where $\ft'=\Lie(T')$, $T'\subseteq \Gad$ is a maximal torus in $\Gad$.\\
Let $x\in\fg$ such that $x\in\fg_2^{\delta!}$ for some $\delta\in \cD_G$. We have 
\begin{align*}
\Ad(\pi(g))(d\pi(x))&=(d(\Int_{\pi(g)})\circ d\pi)(x) =d(\Int_{\pi(g)}\circ\pi)(x)\\
&=d(\pi\circ\Int_g)(x)=d\pi(\Ad(g)(x)).
\end{align*}
Since $x\in\bigoplus_{\substack{\al\in\Phi \\ \delta(\al)=2}}\fg_{\al}$ is nilpotent, we have $\Ad(g)(x)=x$ if and only if $\Ad(\pi(g))(d\pi(x))=d\pi(x)$. This follows, because $\pi$ is a central isogeny and as such is an isomorphism if it is restricted to a closed, connected, unipotent subgroup of $G$, see \cite[Proposition 22.4]{Borel}. Note that every nilpotent element in $\fg_{\ad}$ can be written as $d\pi(y)$ for some nilpotent element $y\in\fg$. For $x\in\fg_2^{\delta!}$ we have $\Ad(g)(x)=x$ only for $g\in G_x\subseteq G_{\geq 0}^{\delta}$. As $\pi$ is an isogeny it induces a bijection $\rho$ between the root systems of $G$ and $\Gad$ where $\pi(U_{\al})=U_{\rho(\al)}$, see \cite[9.6.1]{Springer}. In particular, $\pi(g)\in \left({\Gad}\right)_{\geq 0}^{\tilde{\delta}}$ with $\tilde{\delta}(\al)=\delta(\rho^{-1}(\al))$ and so $d\pi(x)\in\fg_2^{\tilde{\delta!}}$. As $\pi$ restricted to $U_{\al}$ is an isomorphism onto its image, the claim is also true in the other direction.\\
We now show the reduction to semisimple groups. If $G$ is not semisimple we can write $G=[G,G]Z(G)^{\circ}$ and $[G,G]$ is semisimple, see for example \cite[Corollary 8.22]{MalleTesterman}. Let $\fg':=\Lie([G,G])$. If $x\in\fg$ is nilpotent then $x\in\fg'$ and $C_G(x)=C_{[G,G]}(x)Z(G)^{\circ}\subseteq G_{\geq 0}^{\delta}$ if and only if $C_{[G,G]}(x)\subseteq [G,G]_{\geq 0}^{\delta}$ for a weighted Dynkin diagram $\delta$ and therefore for all $\delta\in\cD_G$.
\end{proof}
It is relatively easy to see that each nilpotent piece consists of a union of nilpotent orbits. To verify this, let $x\in\cN_{\fg}^{\btri}$. By the above definition we have $x\in\sigma^{\wtri}$ for a $\wtri \in\btri$. Furthermore, let $y=\Ad(g)(x)$ for some $g\in G$. Then $y\in \Ad(g)(\sigma^{\wtri})$. To see that $\Ad(g)(\sigma^{\wtri})\subseteq\cN_{\fg}^{\btri}$, we write $\sigma^{\wtri}=\fg_2^{\delta!}\oplus\fg_{\geq 3}^{\wtri}$ for some $\delta\in\wtri$.
\begin{lem}\label{ad_on_gd}
Let $g\in G$ and $i\in\Z$. Let $\delta$ be a weighted Dynkin diagram and $\wtri:=\wtri_{\delta}$. Then
\begin{enumerate}[(i)]
\item $\Ad(g)(\fg^{\wtri}_{\geq i})=\fg^{g.\wtri}_{\geq i}$ and
\item $\Ad(g)(\fg^{\delta!}_{2})=\fg^{(g.\delta)!}_{2}$.
\end{enumerate}
\end{lem}
\begin{proof}
\begin{enumerate}[(i)]
\item As $\fg^{\wtri}_{\geq i}=\bigoplus_{j\geq i}\fg^{\wtri}_{j}$ it is enough to show that $\Ad(g)(\fg^{\wtri}_{j})=\fg^{g.\wtri}_{j}$.\\
We have 
\begin{align*}
\hspace*{1.5cm} \Ad(g)(\fg^{\wtri}_{j})&=\Ad(g)(\fg^{\delta}_{j})\\
&=\{\Ad(g)x\in\fg\mid\Ad(\delta(a))x=a^jx\text{ for all }a\in \k^{\+}\}\\
&=\{x\in\fg\mid\Ad(\delta(a)\inv{g})x=a^j\Ad(\inv{g})x\text{ for all }a\in \k^{\+}\}\\
&=\{x\in\fg\mid\Ad(g\delta(a)\inv{g})x=a^jx\text{ for all }a\in \k^{\+}\}\\
&=\fg^{g.\wtri}_{j},
\intertext{ and therefore $\Ad(\fg^{\wtri}_{\geq i})=\bigoplus_{j\geq i}\Ad(\fg^{\wtri}_{j})=\bigoplus_{j\geq i}\fg^{g.\wtri}_{j}=\fg^{g.\wtri}_{\geq i}.$
\item Firstly, we have}
G_{\Ad(g)(x)}&=\{h\in G\mid \Ad(hg)(x)=\Ad(g)(x)\}\\
&=\{h\in G\mid \Ad(\inv{g}hg)(x)=x\}\\
&=\{gh\inv{g}\in G\mid \Ad(h)(x)=x\}\\
&=gG_x\inv{g}.
\intertext{ Additionally, we defined $G_{\geq 0}^{\wtri}$ to be the well-defined parabolic subgroup of $G$ such that $\Lie(G_{\geq 0}^{\wtri})=\fg_{\geq 0}^{\wtri}$. Then $\Lie(G_{\geq 0}^{g.\wtri})=\fg_{\geq 0}^{g.\wtri}$ and by (i) $\fg_{\geq 0}^{g.\wtri}=\Ad(g)(\fg_{\geq 0}^{\wtri})$.\newline
It follows, that $\Lie(G_{\geq 0}^{g.\wtri})=\Ad(g)(\Lie(G_{\geq 0}^{\wtri}))=\Lie(gG_{\geq 0}^{\wtri}\inv{g})$, so $G_{\geq 0}^{g.\wtri}=gG_{\geq 0}^{\wtri}\inv{g}$.\newline
Finally this shows that}
\Ad(g)(\fg^{\delta!}_{2})&=\{\Ad(g)(x)\mid x\in\fg_2^{\delta},~G_x\subseteq G_{\geq 0}^{\delta}\}\\
&=\{\Ad(g)(x)\mid x\in\fg_2^{\wtri},~G_x\subseteq G_{\geq 0}^{\wtri}\}\\
&=\{x\mid x\in\Ad(g)(\fg_2^{\wtri}),~G_{\Ad(\inv{g}x)}\subseteq G_{\geq 0}^{\wtri}\}\\
&=\{x\mid x\in\fg_2^{g.\wtri},~\inv{g}G_{x}g\subseteq G_{\geq 0}^{\wtri}\}\\
&=\{x\in\fg_2^{g.\wtri}\mid G_{x}\subseteq gG_{\geq 0}^{\wtri}\inv{g}\}\\
&=\{x\in\fg_2^{g.\wtri}\mid G_{x}\subseteq G_{\geq 0}^{g.\wtri}\}\\
&=\fg^{(g.\delta)!}_{2}.
\end{align*}
\end{enumerate}
\end{proof}
\begin{cor}
If $x\in\fg$ is nilpotent and $\cO_x$ is the $G$-orbit of $x$ then $x\in\cN_{\fg}^{\btri}$ if and only if $\cO_x\subseteq\cN_{\fg}^{\btri}$.\\
In order to compute the pieces, it is therefore enough to check for each nilpotent orbit in $\fg$ whether a chosen representative of this orbit lies in a given piece.
\end{cor}
\begin{proof}
Let $y\in\cO_x$, i.e. $y=\Ad(g)(x)$ for some $g\in G$.
If $x\in\cN_{\fg}^{\btri}$ it follows by (i) and (ii) in the above lemma that $y\in\Ad(g)(\sigma^{\wtri})=\sigma^{g.\wtri}$ and so $y\in\cN_{\fg}^{\btri}$. The other direction is clear.
\end{proof}
\subsection{Alternative definition of nilpotent pieces}
As remarked upon in the introduction, there exists an alternative definition of nilpotent pieces, given in \cite{ClarkePremet}. This idea arises from the definition of the unipotent pieces, originally defined by Lusztig in \cite{Lusztig1}. We use the same notation as before. Additionally let $\tilde{H}^{\blacktriangle}(\fg)=\bigcup_{\wtri\in\blacktriangle}\fg_{\geq 2}^{\wtri}$. 
\begin{defn}[Nilpotent CP-Pieces, \cite{ClarkePremet}]
Let $$H^{\blacktriangle}(\fg):=\tilde{H}^{\blacktriangle}(\fg)\setminus\bigcup_{\blacktriangle '}\tilde{H}^{\blacktriangle'}(\fg)$$
where $\blacktriangle'$ runs over $G$-orbits such that $\tilde{H}^{\blacktriangle'}(\fg)\subsetneq \tilde{H}^{\blacktriangle}(\fg)$. Then the sets $H^{\blacktriangle}(\fg)$ are the \textbf{nilpotent CP-pieces} in $\fg$.
\end{defn}
One can show that CP-pieces are disjoint and form a partition of the nilpotent variety in $\fg,$ see \cite[Theorem 7]{ClarkePremet}. It is in fact true that the nilpotent pieces defined by Clarke--Premet agree with the nilpotent pieces defined by Lusztig if $G$ is of classical type. Note that the CP-pieces come from the stratification of the nullcone, defined by Hesselink in \cite{Hesselink}, see \cite[Theorem 5]{ClarkePremet}. In \cite{Xue14}, Xue computes the nilpotent pieces in $\fg^*,$ using the definition of Clarke--Premet. As it is not clear whether the nilpotent pieces as introduced by Lusztig agree with the CP-pieces, the nilpotent pieces still have to be computed in $\fg$. 
%
\begin{thm}[{\cite[7.3, Remark 1]{ClarkePremet}}]
If $G$ is simple of classical type $A,B,C$ or $D$ we have $H^{\blacktriangle}(\fg)=\cN_{\fg}^{\blacktriangle}$ for all orbits $\btri$ in any characteristic.
\end{thm}
Again, this problem has not been solved for $G$ of exceptional type in bad characteristic yet, but we hope for the nilpotent pieces and the CP-pieces to agree in all cases.
\subsection{The diagonal cases}\label{diag}
Let $e\in\fg$ be a nilpotent orbit representative, where $\fg:=\Lie(G)$ for a group $G$ of exceptional type. Then $e$ is either in an \textbf{exceptional class}, that is a class that only occurs in bad characteristic, or $e$ is in a non-exceptional class, see \cite[Theorem 9.1 and Tables 22.1.1-22.1.5]{LiebeckSeitz}. Now each class gives rise to a so-called $T$-labelling of the Dynkin diagram which in good characteristic is the weighted Dynkin Diagram $\delta_e$, see again \cite[Theorem 9.1 (ii) and Tables 22.1.1-22.1.5]{LiebeckSeitz}.\\
We will refer to the cases in which we check whether $e_\delta\in\cN_{\fg}^{\btri_{\delta_e}}$ as the \textbf{diagonal cases}.
\begin{lem}\label{diagcase}
Let $e\in\fg$ be a representative of a non-exceptional class, where $\fg$ is the Lie algebra of a group $G$ of exceptional type. Let $\delta_e$ be the weighted Dynkin diagram for the orbit of $e$, as noted above. Then $e\in\cN_{\fg}^{\btri_{\delta_e}}$.
\end{lem}
\begin{proof}
By \cite[Lemma 15.3., Lemma 15.4. and 16.1.1]{LiebeckSeitz} we have $e\in\fg_2^{\delta_e}$ if the class of $e$ is non-exceptional. Furthermore \cite[Theorem 9.1.(ii)(a) and Theorem 16.1.(ii)(a)]{LiebeckSeitz} show that $G_e\subseteq G_{\geq 0}^{\delta_e}$. But then it follows automatically that $e\in\fg_2^{\delta!}$ by definition, so $e\in\cN_{\fg}^{\btri_{\delta_e}}$ as claimed.
\end{proof}
This means that we do not need to do any computational work to decide whether $e\in\cN_{\fg}^{\btri_{\delta_e}}$. In these cases, we will use this lemma instead.
\subsection{The regular piece}\label{regularpiece}
Let $\delta$ be the map corresponding to the weighted Dynkin diagram with weight 2 for every simple root. This is known to always parametrise a nilpotent orbit, see \cite[Chapter 13.1]{Carter85}. We will call this the \textbf{regular diagram} and the corresponding piece the \textbf{regular piece} since such elements are usually called regular. 
In this case, we have $\fg_{\geq 2}^{\wtri}=\bigoplus_{\alpha\in\Phi^+}\fg_{\alpha}$ and $\fg_2^{\wtri}=\bigoplus_{\alpha\in\Pi}\fg_{\alpha}$.
\begin{prop}[The regular piece]\label{regpiece}
Let $\cN_{\fg}^{\btri}$ be the regular piece. Then
$\cN_{\fg}^{\btri}=\cO_x$, the nilpotent orbit of $x=\sum_{\alpha\in\Pi}e_{\alpha}$ with $0\neq e_{\alpha}\in\fg_{\alpha}$ for all $\al\in\Pi$.
\end{prop}
\begin{proof}
Let $\delta\in\wtri\in\btri$ correspond to the regular diagram and suppose that $y=\sum_{\beta\in\Pi}\lambda_{\beta}e_{\beta}\in\fg_2^{\delta}$ such that $\lambda_{\alpha}=0$ for some $\alpha\in\Pi$. We fix this $\al$ for the rest of the proof. For $\gamma,\beta\in\Phi$ let $p_{\beta,\gamma},q_{\beta,\gamma}\in\N$ be so that
\begin{align*}
 \beta+p_{\beta,\gamma}\gamma&\in\Phi &\text{and}& &\beta+(p_{\beta,\gamma}+1)\gamma&\notin\Phi,\\
\beta-q_{\beta,\gamma}\gamma&\in\Phi &\text{and}& &\beta-(q_{\beta,\gamma}+1)\gamma&\notin\Phi,
\end{align*}
see \cite[VI, §1, no. 1.3, Proposition 9]{Bourbaki}.
For $\gamma,\beta\in\Pi$ and $\beta \neq \gamma$ we have $q_{\beta,\gamma}=0$.\\
Let $t\in T$ and $\al\in\Pi$ fixed as above, such that $\beta(t)=(-1)^{p_{\beta,\al}}$ for all $\beta\in\Pi\setminus\{\al\}$ and $\alpha(t)=-1$. Then set $g:=u_{\alpha}(1)tn_{s_{\alpha}}u_{\alpha}(-1)$ by the Bruhat decomposition \ref{bruhatdecomp}. This choice is possible by Dedekind's theorem (\cite[Chapter VIII, §4]{Lang}) and Lemma 16.2 C from \cite{Humphreys75}. As $n_{s_{\alpha}}=u_{\alpha}(1)u_{-\alpha}(-1)u_{\alpha}(1)$ and $u_{-\alpha}(-1)\notin G_{\geq 0}^{\delta}$ we have $n_{s_{\alpha}}\notin G_{\geq 0}^{\delta}$. However, $u_{\alpha}(1),\, t$ and $u_{\alpha}(-1)$ are in  $G_{\geq 0}^{\delta}$, so $g\notin G_{\geq 0}^{\delta}$. We want to show that $\Ad(g)(y)=y$.\\
To compute the action of $\Ad$ on $\fg$ we use Geck's paper \cite{Geck} and note that the action is defined up to sign which will not pose a problem in this case, as we will see in the following calculations.
$\Ad(u_{\alpha}(-1))(y)$: We have 
\begin{align*}
\Ad\left(u_{\al}(-1)\right)(y)&=\Ad(u_{\alpha}(-1))(\sum_{\beta\in\Pi}\lambda_{\beta}e_{\beta})\\
&=\sum_{\beta\in\Pi}\lambda_{\beta}\Ad(u_{\alpha}(-1))(e_{\beta})\\
&=\sum_{\beta\in\Pi}\lambda_{\beta}\sum_{\substack{k\geq 0\\\beta+k\alpha\in\Phi}}\binom{k+q_{\beta,\alpha}}{k}(-1)^k e_{\beta+k\alpha}\\
&=\sum_{\beta\in\Pi}\lambda_{\beta}\sum_{p_{\beta,\alpha}\geq k\geq 0}(-1)^k e_{\beta+k\alpha}=:y'
\intertext{ by \cite[4.10.]{Geck}\vspace*{6pt}\newline
$\Ad(n_{s_{\alpha}})(y')$: We have $s_{\alpha}(\beta)=\beta-(q_{\beta,\alpha}-p_{\beta,\alpha})\alpha=\beta+p_{\beta,\alpha}\alpha$ for $\alpha,\beta\in\Pi$, see \cite[Definition 2.3.]{Geck}. Then}
\Ad(n_{s_{\alpha}})(y')&=\Ad(n_{s_{\alpha}})\big(\hspace*{-8pt}\sum_{\substack{\beta\in\Pi,\\p_{\beta,\alpha}\geq k\geq 0}}\hspace*{-8pt}\lambda_{\beta}(-1)^k e_{\beta+k\alpha}\big)\\
&=\sum_{\substack{\beta\in\Pi,\\p_{\beta,\alpha}\geq k\geq 0}}\lambda_{\beta}(-1)^k \Ad(n_{s_{\alpha}})(e_{\beta+k\alpha})\\
&=\sum_{\substack{\beta\in\Pi,\\p_{\beta,\alpha}\geq k\geq 0}}\lambda_{\beta}(-1)^k (-1)^{p_{\beta,\alpha}-k}e_{\beta+(p_{\beta,\alpha}-k)\alpha}\\
&=\sum_{\substack{\beta\in\Pi,\\p_{\beta,\alpha}\geq k\geq 0}}\lambda_{\beta}(-1)^{p_{\beta,\alpha}}e_{\beta+(p_{\beta,\alpha}-k)\alpha}=:y''
\intertext{by \cite[Lemma 5.4.]{Geck}.\newline
$\Ad(t)(y'')$: An element $t'\in T$ acts on the elements $e_{\alpha}\in\fg$ by $\Ad(t')(e_{\alpha})=\alpha(t)e_{\alpha}$. It follows that with the above choice for $t$ we have}
\Ad(t)(y'')&=\Ad(t)\big(\hspace*{-8pt}\sum_{\substack{\beta\in\Pi,\\p_{\beta,\alpha}\geq k\geq 0}}\hspace*{-8pt}\lambda_{\beta}(-1)^{p_{\beta,\alpha}}e_{\beta+(p_{\beta,\alpha}-k)\alpha}\big)\\
&=\hspace*{-8pt}\sum_{\substack{\beta\in\Pi,\\p_{\beta,\alpha}\geq k\geq 0}}\hspace*{-8pt}\lambda_{\beta}(-1)^{p_{\beta,\alpha}}\Ad(t)(e_{\beta+(p_{\beta,\alpha}-k)\alpha})\\
&=\hspace*{-8pt}\sum_{\substack{\beta\in\Pi,\\p_{\beta,\alpha}\geq k\geq 0}}\hspace*{-8pt}\lambda_{\beta}(-1)^{p_{\beta,\alpha}}\beta(t)\alpha(t)^{p_{\beta,\alpha}-k}e_{\beta+(p_{\beta,\alpha}-k)\alpha}\\
&=\hspace*{-8pt}\sum_{\substack{\beta\in\Pi,\\p_{\beta,\alpha}\geq k\geq 0}}\hspace*{-8pt}\lambda_{\beta}(-1)^{p_{\beta,\al}-k}e_{\beta+(p_{\beta,\alpha}-k)\alpha}=:y'''.
\intertext{Now} 
\sum_{p_{\beta,\alpha}\geq k\geq 0}(-1)^k e_{\beta+k\alpha}&=\sum_{p_{\beta,\alpha}\geq k\geq 0}(-1)^{p_{\beta,\al}-k} e_{\beta+(p_{\beta,\al}-k)\alpha},
\end{align*}
therefore $y'''=y'$ by our choice for $t$.\\
$\Ad(u_{\alpha}(1))(y''')$: Finally,
\begin{align*}
\Ad(u_{\alpha}(1))(y''')=\Ad(u_{\alpha}(1))(y')=\Ad(u_{\alpha}(1))(\Ad(u_{\alpha}(-1))(y))=y.
\end{align*}
This shows that $g\in G_y$, so $y\notin\fg_2^{\delta!}$.

Conversely, let $x=\sum_{\gamma\in\Pi}e_{\gamma}$ as above and  $g\in G\setminus G_{\geq 0}^{\delta}$. We can write $g=u'tn_wu$ by the Bruhat decomposition \ref{bruhatdecomp}.\\
Let $n_w\neq 1$. Then, as $x\in\fg_{2}^{\delta}$, it follows that
\begin{align*}
x'&:=\Ad(u)(x)\in\fg_{\geq 2}^{\delta} \quad \text{ and }\\
x''&:=\Ad(n_w)(x')\in\fg_{\geq 2}^{\delta}\oplus\bigoplus_{\beta\in\Phi^-}\fg_{\beta}
\end{align*}
as there is at least one $\gamma\in\Phi^+$ with $\lambda_{\gamma}\neq 0$ such that $w(\gamma)\in\Phi^-$ for $x'=\sum_{\beta\in\Phi^+}\lambda_{\beta}e_{\beta}$ for $\langle e_{\beta}\rangle=\fg_{\beta}$ and $\lambda_{\beta}\in \k$ (otherwise $w(\Phi^+)\subseteq \Phi^+$ which is only possible for $w=\id$, see for example \cite[Theorem A.22]{MalleTesterman}). In particular, $\Ad(n_w)(x')\notin\fg_{\geq 2}^{\delta}$. But then $\Ad(u't)(x'')\notin\fg_{2}^{\delta}$ and so there is no element
$g\in G\setminus G_{\geq 0}^{\delta}$ that centralizes $x$.\\
As we can find $t\in T$ such that $\lambda_{\beta}=\beta(t)$ for all $\beta\in\Pi$ and $\lambda_{\beta}\in \k^{\+}$ (again by \cite[Chapter VIII, §4]{Lang} and \cite[Lemma 16.2 C]{Humphreys75}), it follows that $\fg_{2}^{\delta!}=\{ \sum_{\beta\in\Pi}\lambda_{\beta}e_{\beta}\mid \lambda_{\beta}\neq 0,\text{ for all }\beta\in\Pi\}$. This proves the claim.
\end{proof}
\section{A Computational Approach}\label{computation}
In order to compute the pieces as defined in Definition \ref{nilpotentpieces}, we will first present a few results on the action of $G$ on its Lie algebra and in particular on the sets $\fg_2^{\delta!}$. Having computed the pieces in the Lie algebras of exceptional type, it should be within reach to prove that Lusztig's nilpotent pieces agree with the CP-pieces.
\subsection{Computing the action of \texorpdfstring{$\Ad$}{}}\label{Ad-action}
As mentioned in section \ref{regularpiece} we can compute the action of $G$ via $\Ad$ on $\fg$ up to sign by following \cite{Geck}. We will consider the actions of a unipotent element in $B$, an element of the torus and a Weyl group representative respectively. As the action of $\Ad$ on $\fg$ is linear, it is enough to examine this action on the elements $e_{\alpha}$ where $\langle e_{\alpha}\rangle=\fg_{\alpha}$ for $\alpha\in\Phi$.
\begin{enumerate}
\item Action of an element $u_{\alpha}(c_{\alpha})\in U=\prod_{\beta\in\Phi^+}U_{\beta}$ on $\fg$ where $\alpha\in\Phi^+$ and $c_{\alpha}\in \k$. Let $x\in\fg$. Then
\begin{align*}
\Ad(u_{\alpha}(c_{\alpha}))(x)&=\exp(c_{\alpha}e_{\alpha})(x)\\
&=\sum_{i\geq 0}\frac{c_{\alpha}^i}{i!}(\ad(e_{\alpha}))^i(x),
\end{align*}
where we set $\ad(e_{\alpha})(x)=[e_{\alpha},x]$ and $(\ad(e_{\alpha}))^i(x)=\ad((\ad(e_{\alpha}))^{i-1}(x))$ for $i> 1$, by \cite[4.10. and section 5]{Geck}.\\
Note that we are dividing by $i!$ in this formula. In particular, we need to be careful if the characteristic of $\k$ is finite: In this case we will first compute $\Ad(u_{\alpha}(c_{\alpha}))(x)$ in characteristic 0 and then reduce modulo the characteristic of $\k$. This is possible since the coefficients $\frac{c_{\alpha}^i}{i!}$ are integers, see \cite[4.10. and Corollary 5.6.]{Geck}
\item Action of a representative $n_w\in N_G(T)\subseteq G$ of the element $w\in W$. First, define the map $n_{\alpha}(c):\fg\ra\fg$ for $c\in \k^{\+}$ and $\alpha\in\Pi$ by $n_{\alpha}(c):=\exp(ce_{\alpha})\exp(\inv{-c}e_{-\alpha})\exp(ce_{\alpha})$. There exist elements $h_i\in\ft$ for $i\in\{1,\ldots,|\Pi|\}$ that form a basis of $\ft$, i.e. $\ft=\langle h_i\mid i\in\{1,\ldots,|\Pi|\}\rangle$ and $\Pi=\{\alpha_1,\ldots,\alpha_n\}$ as in \cite[Definition 5.2.]{Geck}. Then
\begin{align*}
n_{\alpha_i}(c)(h_j)&=h_j-|(\alpha_i,\alpha_j^{\vee})|h_i,\\[0.2cm]
n_{\alpha_i}(c)(e_{\alpha})&=\begin{cases}
c^{-2}e_{-\alpha_i} &\alpha=\alpha_i,\\
c^2e_{\alpha_i} &\alpha=-\alpha_i,\\
-(-1)^{q_{\al,\al_i}+1}c^{-(\alpha,\alpha_i^{\vee})}e_{s_i(\alpha)}  &\mbox{else}.
\end{cases}
\end{align*}
Here $(\alpha_i,\alpha_j^{\vee})$ is the $(i,j)$-th entry of the Cartan matrix of $G$ and $q_{\al,\al_i}$ is as in the proof of Proposition \ref{regpiece} (see \cite[Lemma 5.4]{Geck}).
In this case we have $\Ad(n_{s_{\alpha}})(x)=n_{\alpha}(1)(x)$ for all $x\in\fg$.
\item Action of an element of the torus $T$ on $\fg$.  For every element in $T$ we can find $c\in \k^{\+}$ and $\alpha\in\Pi$ such that the action of this element can be represented by the map $h_{\alpha}(c):=n_{\alpha}(c)n_{\alpha}(-1)$, see \cite[Theorem 12.1.1]{Carter72}.
\end{enumerate}
\subsection{Practical aspects}
Let $x\in\fg$ be a nilpotent element and $\delta\in\cD_G$ correspond to a weighted Dynkin diagram. As we are interested in the set $\fg_2^{\delta}$ we will define the \enquote{parts} of certain nilpotent elements $x\in\fg$ that lie in it.
\begin{defn}[$\fg_i^{\delta}$-part]
Let $x\in\fg$ be a nilpotent element such that we can write $x=\sum_{\alpha\in\Phi}\lambda_{\alpha}e_{\alpha}$ where each $e_{\alpha}$ generates the subalgebra $\fg_{\alpha}$ for $\alpha\in\Phi$. Then we define for $i\in\Z$
\begin{align*}
[x]_{\fg_i^{\delta}}:=\sum_{\delta(\alpha)=i}\lambda_{\alpha}e_{\alpha}
\end{align*}
as the $\fg_i^{\delta}$-part of $x$. In particular, we have $x=[x]_{\fg_i^{\delta}}+\sum_{\delta(\alpha)\neq i}\lambda_{\alpha}e_{\alpha}$ for all $i\in\Z$.
\end{defn}
We want to decide whether $x\in\cN_{\fg}^ {\btri}$ for a given orbit $\btri$. This can be done by checking two things:
\begin{enumerate}
\item Check whether there exists some $g\in G$ such that $\Ad(g)(x)\in\fg_{\geq 2}^{\delta}$. If not, then $x$ cannot lie in $\cN_{\fg}^{\btri}$.
\item If (1) is fulfilled, we need to take a closer look at the $\fg_2^{\delta}$-part of $\Ad(g)(x)$. 
By the above definition we have $$\Ad(g)(x)=[\Ad(g)(x)]_{\fg_2^{\delta}}+\sum_{i\geq 3}[\Ad(g)(x)]_{\fg_i^{\delta}}$$ and so $\Ad(g)(x)\in\sigma^ {\delta}\subseteq\cN_{\fg}^{\btri}$ if $[\Ad(g)(x)]_{\fg_2^{\delta}}\in\fg_{2}^ {\delta!}$.
Conversely, suppose that there is no $g\in G$ such that $[\Ad(g)(x)]_{\fg_2^{\delta}}\in\fg_{2}^ {\delta!}$. Then $x\notin\cN_{\fg}^ {\btri}$ as otherwise we would have $x\in\sigma^{h.\delta}$ for some $h\in G$, $h.\delta\in\wtri\in\btri$ and by Lemma \ref{ad_on_gd} we would have $\Ad(\inv{h})(x)\in\sigma^{\delta}$. In particular $[\Ad(\inv{h})(x)]_{\fg_2^{\delta}}\in\fg_{2}^{\delta!}$.\\
We can therefore in fact concentrate on deciding whether\\ $[\Ad(g)(x)]_{\fg_2^{\delta}}\in\fg_{2}^{\delta!}$.
\end{enumerate}
By the Bruhat decomposition, every element $g\in G$ can be written uniquely as $g=u'tn_wu$, see Definition \ref{bruhatdecomp}.
Furthermore, every element $u\in\prod_{\alpha\in\Phi^+}U_{\alpha}$ can be written as $u=\prod_{\alpha\in\Phi^+}u_{\alpha}(c_{\alpha})$ for our fixed maps $u_{\alpha}:\k^ {\+}\ra U_{\alpha}$ and $c_{\alpha}\in \k^ {\+}$ as in \eqref{ualpha}. Therefore, the elements in $G$ can be parametrized by $w\in W$, $c_{\alpha},c'_{\alpha}\in \k^{\+}$, such that $u=\prod_{\alpha\in\Phi^+}u_{\alpha}(c_{\alpha})$, $u'=\prod_{\substack{\alpha\in\Phi^+\\w.\al\in\Phi^-}}u_{\alpha}(c'_{\alpha})$, and $t\in T$.\\
Let $g_w:=u'tn_wu\in G$ as above. Since $W$ is finite, it can be possible to decide whether $\Ad(g_w)(x)\in\fg_{\geq 2}^{\delta}$ for each $w\in W$. In practice, we might encounter restrains such as memory space or time, that might make it difficult to compute the whole Weyl group $W$ or decide whether $\Ad(g_w)(x)\in\fg_{\geq 2}^{\delta}$.\\
 Following section \ref{Ad-action} we can compute 
\begin{align}\label{ad(gw)}
\Ad(g_w)(x)=\Ad(u'tn_wu)(x)=\sum_{\alpha\in\Phi}\lambda_{\alpha}e_{\al}+\sum_{i=1}^ {|\Pi|}\mu_{i}h_{i}
\end{align} where the $h_i$ form a basis of $\ft$ and $\lambda_{\al},\mu_i\in\k$ depend on $g_w$ and can be determined by the rules in section \ref{Ad-action}. To check whether we can choose $g_w$ such that $\Ad(g_w)(x)\in\fg_{\geq 2}^{\delta}$ we need to solve the system of non-linear equations given by \eqref{ad(gw)}:
\begin{align}\label{nonlineqs}
\begin{split}
\lambda_{\alpha}&=0\quad\text{for all } \alpha\in\Phi\text{ with } \delta(\al)\leq 1,\\
\mu_i&=0\quad \text{for all } i\in\{1,\ldots,|\Pi|\},
\end{split}
\end{align}
where $\lambda_{\al},\mu_i$ are polynomials in the variables $c_{\alpha},c'_{\alpha}$ determined by $u=\prod_{\alpha\in\Phi^+}u_{\alpha}(c_{\alpha})$ and $u'=\prod_{\alpha\in\Phi^+}u_{\alpha}(c'_{\alpha})$. The entries in $t$ are as given in section \ref{Ad-action} (3).
Using Groebner bases, we can decide whether the system can be solved and determine a solution.\\
If $\Ad(g_w)(x)\in\fg_{\geq 2}^{\delta}$ we continue with step (2), otherwise we check whether $\Ad(g_w)(x)\in\fg_{\geq 2}^{\delta}$ for the next $w\in W$.\\
This process can be simplified by considering Weyl group elements $w\in W$ of \textbf{weight} $0$.
\begin{defn}
Let $\delta\in\cD_G$ be a map arising from a weighted Dynkin diagram and let $w=s_{\al_1}\cdots s_{\al_r}\in W$ for some $r\in\N$ and $\al_1,\ldots,\al_r\in\Pi$. Then we say that $w$ has \textbf{weight} $0$ if $\delta(\al_1)=\cdots =\delta(\al_r)=0$. These elements form a subgroup of $W$ which we will denote by $W_0^{\delta}:=\langle s_{\al}\mid \al\in\Pi,~\delta(\al)=0\rangle$.
\end{defn}
We can show that $n_w$ fixes the set $\fg_{\geq 2}^{\delta}$ if $w$ has weight $0$. It turns out that an even stronger result is true, as stated in the following lemma.
\begin{lem}\label{checkg2d}
Let $\delta\in\cD_G$ be a map arising from a weighted Dynkin diagram.
Let $g_w=u'tn_wu\in G$ as in Definition \ref{bruhatdecomp} and $y\in\fg_{\geq 2}^{\delta}$. Then:
\begin{enumerate}[(i)]
\item The element $g_w$ is contained in $G_{\geq 0}^{\delta}$ if and only if $w\in W_0^{\delta}$.
\item If $g_w\in G_{\geq 0}^{\delta},$ then $Ad(g_w)(y)\in\fg_{\geq 2}^{\delta}$.
\end{enumerate}
In particular, $\Ad(g_w)(y)\in\fg_{\geq 2}^{\delta}$ if and only if $\Ad(n_wu)(y)\in\fg_{\geq 2}^{\delta}$.
\end{lem}
\begin{proof}~
\begin{enumerate}[(i)]
\item Write $g_w=u'tn_wu$ using the Bruhat decomposition with the notation from Definition \ref{bruhatdecomp}. Clearly, any elements $u\in U$ and $t\in T$ are contained in $G_{\geq 0}^ {\delta},$ so $u,u',t\in G_{\geq 0}^ {\delta}$ and therefore $g_w\in G_{\geq 0}^ {\delta}$ if and only if $n_w\in G_{\geq 0}^ {\delta}$.
Furthermore, let $s_{\alpha}\in W$ be the reflection for the root $\alpha\in\Phi$. Then we can choose $n_{s_{\alpha}}=u_{\alpha}(1)u_{-\alpha}(-1)u_{\alpha}(1),$ see Section \cite[1.9, p.19]{Carter85}. Note also that $s_{\al}=s_{-\al}$ in $W,$ therefore $n_{s_{\al}}=n_{s_{-\al}}t$ for some $t\in T$. It follows that if $\delta(\al)\geq 0$ we have $u_{\alpha}(1)\in G_{\geq 0}^{\delta}$ and therefore $n_{s_{\alpha}}=u_{\alpha}(1)u_{-\alpha}(-1)u_{\alpha}(1)\in G_{\geq 0}^{\delta}$ if and only if $\delta(\al)=0$. If, on the other hand, $\delta(\al)\leq 0,$ we have $u_{-\alpha}(1)\in G_{\geq 0}^{\delta}$ as well as $t\in G_{\geq 0}^{\delta}$ and therefore $n_{s_{\alpha}}=n_{s_{-\al}}t=u_{-\alpha}(1)u_{\alpha}(-1)u_{-\alpha}(1)t\in G_{\geq 0}^{\delta}$ if and only if $\delta(\al)=0$. This means that for $w=s_{\alpha_1}\cdots s_{\alpha_r}$ we have $n_w\in G_{\geq 0}^{\delta}$ if $\delta(\alpha_1)=\cdots=\delta(\alpha_r)=0,$ i.e.\ if $w$ has weight $0$.\\
 Conversely, let $w=s_{\alpha_1}\cdots s_{\alpha_r}$ be a reduced form of $w\in W$. Since $\delta$ arises from a weighted Dynkin diagram, we have $\delta(\al)\geq 0$ for all $\al\in\Phi^+$. Then $G_{\geq 0}^{\delta}=P_I=BW_IB$ is a standard parabolic subgroup, where $\Phi_I=\{\al\in\Phi\mid \delta(\al)=0\}$ and $W_ I=\langle s_{\al}\mid \al\in\Phi_I\rangle$. Clearly, $W_0^{\delta}=W_I$. By the uniqueness of the Bruhat decomposition in Theorem \ref{bruhatdecomp}, $n_w\in G_{\geq 0}^{\delta}$ if and only if $w\in W_I$. Therefore, $\delta(\alpha_i)= 0$ for all $i\in\{1,\ldots,r\}$ if $n_w\in G_{\geq 0}^{\delta}$.  This proves the claim. 
\item To see that $\Ad(g)(y)\in\fg_{\geq 2}^{\delta},$ it is enough to show that  for each $e_{\beta}\in \fg_{\beta}$ with $\beta\in\Phi$ such that $\delta(\beta)\geq 2$ the following claims are true:
\begin{enumerate}[(1),itemsep=5pt]
\item $\Ad(u_{\al}(c_{\al}))(e_{\beta})\in\fg_{\geq 2}^{\delta}$ for all $c_{\al}\in\k,$
\item $\Ad(t)(e_{\beta})\in\fg_{\geq 2}^{\delta}$ for all $t\in T,$ and 
\item$\Ad(n_w)(e_{\beta})\in\fg_{\geq 2}^{\delta}$ for all $n_w\in G_{\geq 0}^{\delta}$.
\end{enumerate} \vspace*{0.2cm}
\begin{enumerate}
\item[To (1):]
We have
\begin{align*}
\Ad(u_{\alpha}(t_{\alpha}))(e_{\beta})=\sum_{\substack{\beta+k\alpha\in\Phi\\ k\geq 0}}t_{\alpha}^kc_{k,\alpha,\beta}e_{\beta+k\alpha}\in\fg_{\geq 2}^{\delta},
\end{align*}
for some $t_{\alpha},c_{k,\alpha,\beta}\in \k$ and $\delta(\beta+k\alpha)=\delta(\beta)+k\delta(\alpha)\geq 2$ as both $k\geq 0$ and $\delta(\alpha)\geq 0$ for $\al\in\Phi$.
\item[To (2):]
For $t\in T$ we have $\Ad(t)(e_{\beta})=\beta(t)e_{\beta}\in\fg_{\beta},$ so the action of $T$ stabilises $\fg_{\geq 2}^ {\delta}$.
\item[To (3):] By Section \ref{Ad-action}(2) we have $\Ad(n_w)(\fg_{\beta})=\fg_{w(\beta)}$. As before, write $w=s_{\alpha_1}\cdots s_{\alpha_r}$ where $\delta(\alpha_1)=\cdots=\delta(\alpha_r)=0$ and $\alpha_i\in\Pi$ for all $i$ (since $\delta(\beta)\geq 2,$ the reflection $s_{\beta}$ is not in this product).
For two roots $\alpha,\beta\in\Phi$ we have $s_{\alpha}(\beta)=\beta+k_{\al,\beta}\al$ for some $k_{\al,\beta}\in\Z$. If $\delta(\al)=0,$ then $\delta(s_{\alpha}(\beta))=\delta(\beta)$. Iteratively, we get $\delta(w(\beta))=\delta(\beta)$ and so $\Ad(n_w)(\fg_{\beta})=\fg_{w(\beta)}\in\fg_{\geq 2}^{\delta}$.
\end{enumerate}
Finally, let $w\in W$ be an arbitrary element of the Weyl group and suppose that  $\Ad(g_w)(y)=x\in\fg_{\geq 2}^{\delta}$.\\
 As $u't\in G_{\geq 0}^{\delta},$ so is $\inv{(u't)}\in G_{\geq 0}^{\delta}$ and therefore by the above calculations\\ $\Ad(n_wu)(y)=\Ad(\inv{(u't)})(x)\in\fg_{\geq 2}^{\delta}$.\\
Conversely, suppose that $\Ad(n_wu)(y)=x\in\fg_{\geq 2}^{\delta}$. Since, as before, $u't\in G_{\geq 0}^{\delta},$ it follows that $\Ad(u'tn_wu)(y)=\Ad(u't)(x)\in\fg_{\geq 2}^{\delta}$. This proves the claim.
\end{enumerate}
\vspace*{-0.5cm}
\end{proof}
In particular, this shows that for $g\in G_{\geq 0}^{\delta}$ we do not have to check whether $\Ad(g)(y)$ is in $\fg_{\geq 2}^{\delta}$, as the above lemma states that this is the case whenever $y\in\fg_{\geq 2}^{\delta}$. Additionally, let $u=u'u''$ for $u\in U,\,u'\in\prod_{\substack{\alpha\in\Phi^+\\w.\alpha\in\Phi^+}}U_{\alpha}$ and $u''\in U_w^-$. Then $n_wu=n_wu'\inv{n_w}n_wu''$ and $n_wu'\inv{n_w}\in U\subseteq G_{\geq 0}^{\delta}$, see \cite[Theorem 8.17.(e)]{MalleTesterman}. This means, it is even enough to check whether $\Ad(n_wu'')(y)\in\fg_{\geq 2}^{\delta}$.\\
Another direct consequence is the following: Suppose we already know that $\Ad(n_wu)(y)\in\fg_{\geq 2}^{\delta}$ (or $\Ad(n_wu)(y)\notin\fg_{\geq 2}^{\delta}$) for some $w\in W$. Then for all $w'\in W$ with weight $0$ we have 
$\Ad(n_{w'}n_wu)(y)=\Ad(n_{w'w}u)(y)\in\fg_{\geq 2}^{\delta}$ (resp. $\Ad(n_{w'w}u)(y)\notin\fg_{\geq 2}^{\delta}$).\\
There is a similar result when we consider the $\fg_2^{\delta}$-part and the set $\fg_2^{\delta!}$.
\begin{lem}\label{nw_u}
Let $x\in\fg$ be a nilpotent element and $w\in W$, $u\in U_w^-$ with $\Ad(n_wu)(x)\in\fg_{\geq 2}^{\delta}$ and $[\Ad(n_wu)(x)]_{\fg_2^{\delta}}\notin\fg_2^{\delta!}$.\\
Then we have $[\Ad(n_{w'w}u)(x)]_{\fg_2^{\delta}}\notin\fg_2^{\delta!}$ for any $w'\in W_0^{\delta}$.
\end{lem}
\begin{proof} 
As $\Ad(n_{w}u)(x)\in\fg_{\geq 2}^{\delta}$ it follows that $\Ad(n_{w'w}u)(x)\in\fg_{\geq 2}^{\delta}$ for any $w'\in W_0^{\delta}$.\\
For $g\in G_{0}^{\delta}$ we have (by a similar calculation as in the proof of Lemma \ref{checkg2d}) $\Ad(g)(\fg_i^{\delta})=\fg_i^{\delta}$ for each $i\in\Z$. Then 
\begin{align*}
\hspace*{1cm} [\Ad(n_{w'w}u)(x)]_{\fg_2^{\delta}}=\Ad(n_{w'})([\Ad(n_wu)(x)]_{\fg_2^{\delta}})
\end{align*}
since $n_{w'}\in G_0^{\delta}$.\\
For easier notation let $y:=[\Ad(n_wu)(x)]_{\fg_2^{\delta}}$ and $y':=[\Ad(n_{w'w}u)(x)]_{\fg_2^{\delta}}$. Suppose that there exists $g\in G\setminus G_{\geq 0}^{\delta}$ with $\Ad(g)(y)=y$. Then\\
 $\Ad(n_{w'}g\inv{n_{w'}})(y')=y'$. As $n_{w'}\in G_{\geq 0}^{\delta}$ and $g\notin G_{\geq 0}^{\delta}$ we have $n_{w'}g\inv{n_{w'}}\notin G_{\geq 0}^{\delta}$ and so $y'\notin \fg_2^ {\delta!}$.
\end{proof}
This lemma simplifies the calculations further: Suppose we already know that $\Ad(n_wu)(x)\notin\fg_{\geq 2}^{\delta!}$ for some $w\in W$. Then for all $w'\in W$ with weight $0$ we have 
$\Ad(n_{w'}n_wu)(x)=\Ad(n_{w'w}u)(x)\notin\fg_{\geq 2}^{\delta!}$. Therefore, we do not have to check $\Ad(n_wu)(x)$ if $w\in W$ is of the form $w=w'w''$, $w'\in W_0^{\delta}$, and we already know that $\Ad(n_{w''}u)(x)\notin\fg_{\geq 2}^{\delta!}$.\\
However, we do need some further results to justify focusing on elements of the form $\Ad(n_wu)(x)$.
\begin{lem}\label{nw.u.x}
Let $x\in\fg_{\geq 2}^{\delta}$, $u=\prod_{\beta\in\Phi^+}u_{\beta}(t_{\beta})\in U$, $t_{\beta}\in\k$ and $\delta(\beta)\geq 0$ for all $\al\in\Phi^+$. Let $\tilde{u}:=\prod_{\substack{\beta\in\Phi^+\\ \delta(\beta)=0}}u_{\beta}(t_{\beta})\in U$.\\
 Then $[\Ad(u)(x)]_{g_2^{\delta}}=\Ad(\tilde{u})([x]_{g_2^{\delta}})$.
\end{lem}
\begin{proof}
We have $x=\sum_{\substack{\al\in\Phi^+\\ \delta(\al)\geq 2}}\lambda_{\al}e_{\al}$. Furthermore let $\beta\in\Phi^+$ and consider $\Ad(u_{\beta}(c))(x)$ for some $c\in \k$. We have
\begin{align*}
\Ad(u_{\beta}(c))(x)&=\sum_{\substack{\al\in\Phi^+\\ \delta(\al)\geq 2}}\lambda_{\al}\sum_{\substack{k\geq 0,\\\al+k\beta\in\Phi}}c_{\al,k,\beta}e_{\al+k\beta},
\end{align*}
and 
\begin{align*}
\Ad(u_{\beta}(c))([x]_{g_2^{\delta}})&=\sum_{\substack{\al\in\Phi^+\\ \delta(\al)= 2}}\lambda_{\al}\sum_{\substack{k\geq 0,\\\al+k\beta\in\Phi}}c_{\al,k,\beta}e_{\al+k\beta},
\end{align*}
for some $c_{\al,k,\beta}\in\k$. As $\beta\in\Phi^+$ we have either $\delta(\beta)=0$ or $\delta(\beta)>0$.\\
In the first case $\delta(\al+k\beta)=\delta(\al)$ for all $k\in\Z_{\geq 0}$ and in the second case we have $\delta(\al+k\beta)=\delta(\al)+k\delta(\beta)>\delta(\al)$ for all $k\in\Z_{\geq 0}$.\\
So $[\Ad(u_{\beta}(c))(x)]_{g_2^{\delta}}=\Ad(u_{\beta}(c))([x]_{g_2^{\delta}})\in\fg_2^{\delta}$ if $\delta(\beta)=0$.\\
Suppose $\delta(\beta)\neq 0$. Then 
\begin{align*}
\Ad(u_{\beta}(c))(x)=&\sum_{\substack{\al\in\Phi^+ \\ \delta(\al)=2}}\big(\lambda_{\al}e_{\al}+\lambda_{\al}\sum_{\substack{k\geq 0,\\\al+k\beta\in\Phi}}c_{\al,k,\beta}e_{\al+k\beta}\big)\\
&+\sum_{\substack{\al\in\Phi^+ \\ \delta(\al)> 2}}\lambda_{\al}\sum_{\substack{k\geq 0,\\\al+k\beta\in\Phi}}c_{\al,k,\beta}e_{\al+k\beta},
\end{align*}
and so $[\Ad(u_{\beta}(c))(x)]_{g_2^{\delta}}=\sum_{\substack{\al\in\Phi^+ \\ \delta(\al)=2}}\lambda_{\al}e_{\al}=[x]_{g_2^{\delta}}.$\\
The claim follows inductively.
\end{proof}
\begin{prop}
We use the same notation as in Lemma \ref{nw.u.x}. Suppose that $[\Ad(n_wu)(x)]_{\fg_2^{\delta}}\notin\fg_2^{\delta!}$. Then
\begin{enumerate}[(i)]
\item  for all $u'\in U$ and $t\in T$ it follows that $[\Ad(u'tn_wu)(x)]_{\fg_2^{\delta}}\notin\fg_2^{\delta!}$ as well.
\item For $w_0\in W_0^{\delta}$ we have $[\Ad(n_{w_0w}u)(x)]_{\fg_2^{\delta}}\notin\fg_2^{\delta!}$.
\end{enumerate} 
As a direct consequence we have the following: Let $w$ be in the right  transversal of $W_0$ in $W$. Then it is enough to check whether $[\Ad(n_wu)(x)]_{\fg_2^{\delta}}$ is in $\fg_2^{\delta!}$ in order to decide whether $[\Ad(u'tn_{w_0w}u)(x)]_{\fg_2^{\delta}}$ is in $\fg_2^{\delta!}$ for all $u'\in U$, $t\in T$ and $w_0\in W_0^{\delta}$.
\end{prop}
\begin{proof} Let $y:=[\Ad(n_wu)(x)]_{\fg_2^{\delta}}$. 
\begin{enumerate}[(i)]
\item
As $t\in T$ we have $[\Ad(tn_wu)(x)]_{\fg_2^{\delta}}=\Ad(t)([\Ad(n_wu)(x)]_{\fg_2^{\delta}})=\Ad(t)(y)$. By the above lemma we can write $[\Ad(u'tn_wu)(x)]_{\fg_2^{\delta}}=\Ad(\tilde{u}t)(y)$.\\
As $y\notin\fg_2^{\delta!}$ there exists $g\in G\setminus G_{\geq 0}^{\delta}$ such that $\Ad(g)(y)=y$. Then $h:=\tilde{u}tg(\tilde{u}t)^{-1}\in G\setminus G_{\geq 0}^{\delta}$ and $\Ad(h)(\Ad(\tilde{u}t)(y))=\Ad(\tilde{u}t)(y)$ which proves (i).
\item Since $w_0\in W_0^{\delta}$ we have $n_{w_0}\in G_{0}^{\delta}$ and therefore $[\Ad(n_{w_0w}u)(x)]_{\fg_2^{\delta}}=\Ad(n_{w_0})(y)$ as seen in the proof of Lemma \ref{nw_u}. By the same argument as in (i) there exists $g\in G\setminus G_{\geq 0}^{\delta}$ such that $\Ad(g)(y)=y$ and therefore $\Ad(n_{w_0}gn_{w_0}^{-1})([\Ad(n_{w_0w}u)(x)]_{\fg_2^{\delta}})=$ $[\Ad(n_{w_0w}u)(x)]_{\fg_2^{\delta}}$ $=\Ad(n_{w_0})(y)$. As $n_{w_0}gn_{w_0}^{-1}\notin G_{\geq 0}^{\delta}$ it follows that $[\Ad(n_{w_0w}u)(x)]_{\fg_2^{\delta}}\notin\fg_2^{\delta!}$.
\end{enumerate}
\end{proof}
\subsection{Description of the algorithm to compute the pieces}
 To determine which nilpotent orbits make up a piece $\cN_{\fg}^{\btri}$ we fix a map $\delta\in\wtri\in\btri$ and check whether $\cO \subseteq \cN_{\fg}^{\btri}$ for each nilpotent orbit $\cO$ in $\fg$. We will use the results from the previous sections to simplify the computations and proceed in several steps. The computations were carried out using the computer programme Magma, \cite{Magma}.
 We fix the following notation: Let $x\in\fg$ denote a nilpotent orbit representative, where $x$ is chosen such that $x=\sum_{\al\in\Phi^+}\lambda_{\al}e_{\al}$ for suitable $\lambda_{\al}\in\k$. Define tuples
\begin{align*}
c&:=(c_{\al})_{\al\in\Phi^+},\\
c'&:=c_w':=(c_{\al}')\substack{\al\in\Phi^+\\ \inv{w}.\al\in\Phi^-} \quad\text{for }w\in W,\text{ and}\\
d&:=(d_{\al})_{\al\in\Pi},
\end{align*}
where the entries are elements in a function field over $\k$. Let $g_w(c,d,c'):=u'(c')t(d)n_wu(c)$ be the Bruhat decomposition of an element in $G$ where 
\begin{align*}
u'(c')&:=\prod_{\substack{\al\in\Phi^+\\ \inv{w}.\al\in\Phi^-}}u_{\al}(c_{\al}'),\\
t(d)&:=\prod_{\al\in\Pi}h_{\al}(d_{\al}),\text{ and}\\ u(c)&:=\prod_{\al\in\Phi^+}u_{\al}(c_{\al})
\end{align*}
with the $h_{\al}$ defined as in Section \ref{Ad-action} and for a fixed ordering on $\Phi^+$.
We write
 $$x_w(c,d,c'):=\Ad(u'(c')t(d)n_wu(c))(x)$$
 for the action of $g_w(c,d,c')$ on $x$. Finally, let $\lambda_{w,\beta}(c,d,c')\in\k$ for $\beta\in\Phi$ and $\mu_{w,i}(c,d,c')\in\k$ for $i\in\{1,\ldots,|\Pi|\}$ be defined as in \eqref{ad(gw)}:
\begin{align*}
\Ad(g_w(c,d,c'))(x)=\sum_{\beta\in\Phi}\lambda_{w,\beta}(c,d,c')e_{\beta}+\sum_{i=1}^ {|\Pi|}\mu_{w,i}(c,d,c')h_{i}.
\end{align*}
 \begin{enumerate}[itemsep=5pt,label=Step \arabic*: , ref={Step~\arabic*}]
\item \label{checkifing2dg2d} Let $\delta$ arise from a fixed weighted Dynkin diagram and let $x\in\fg$ be a nilpotent element such that $x=\sum_{\al\in\Phi^+}\lambda_{\al}e_{\al}$ for suitable $\lambda_{\al}\in\k$. The element $x$ is a representative of a nilpotent orbit, denoted by $\cO_x$.\\
Fix $w\in W$. For an arbitrary element $g_w(c,d,c')\in Bn_wB$, we want to check if we have $\Ad(g_w(c,d,c'))(x)=x_w(c,d,c')\in\fg_{\geq 2}^{\delta}$. In this manner, it is possible to compute the $\Ad$-action of every element in $G$ on $x$.\\
By Lemma \ref{checkg2d}, it is enough to check if $\Ad(n_w\tilde{u})(x)\in\fg_{\geq 2}^{\delta}$ where 
$$\tilde{u}=u_{\al_1}(\tilde{c}_{\al_1})\cdots u_{\al_r}(\tilde{c}_{\al_r})$$
denotes an element in $U$, parametrised by the elements $\tilde{c}_{\al_1},\ldots, \tilde{c}_{\al_r}\in\k$ for some $r\in\N$ and $\al_1,\ldots,\al_r\in\Phi^+$ such that $w(\al_i)\notin\Phi^+$ for all $i=1,\ldots,r$. From now on we will use the notation $c_{\al_i}$ instead of $\tilde{c}_{\al_i}$ for easier reading and let $x_w(c):=x_w(c_{\al_1},\ldots, c_{\al_r}):=\Ad(n_w\tilde{u})(x)$ be the element that depends on the $c_{\al_i}$.\\
In order to check if $\Ad(n_w\tilde{u})(x)\in\fg_{\geq 2}^{\delta}$, we first compute the action of $n_w\tilde{u}$ on $x$ as in Section \ref{Ad-action}. Following this, the resulting system of non-linear equations for the $c_{\al_i}$  (see \eqref{nonlineqs}) can be solved by computing a Gröbner basis of this system. For this we use the standard algorithm in Magma, see \cite[Section 112.4.3]{MagmaHandbook} with a reverse lexicographical ordering. In order to speed this process up, we will first check for variables that occur in linear equations and solve for those variables. We will apply this every time we compute a Gröbner basis. Note that in general the solution (if there exists one) will still depend on some of the $c_{\al_i}$. This means that we will have to check whether $[\Ad(n_w\tilde{u})(x)]_{\fg_{2}^{\delta}}\in\fg_{2}^{\delta!}$ for each solution.
\item Suppose there exists a solution such that $\Ad(n_w\tilde{u})(x)\in\fg_{\geq 2}^{\delta}$. As mentioned in \ref{checkifing2dg2d} this solution will still depend on some of the $c_{\al_i}$. Let $y_w:=x_w(\varepsilon_1,\ldots, \varepsilon_r),$ $\varepsilon_i\in\{0,1,\ldots,\chara(\k)-1\}$ be the element arising from a particular solution where we replaced the $c_{\al_i}$ in $\tilde{u}=u_{\al_1}(c_{\al_1})\cdots u_{\al_r}(c_{\al_r})$ by the  solution  in which we set the remaining $c_{\al_i}$ to zero if possible or otherwise another fixed element in the prime field of $\k$. This depends on which solution is possible, so that we do not divide by zero when replacing the $c_{\al_i}$.
\item \label{cut2} Let $x'_w(c):=[x_w(c)]_{\fg_2^{\delta}}$ and $y'_w:=[y_w]_{\fg_2^{\delta}}$ by setting the coefficients of the basis elements not contained in $\fg_2^{\delta}$ to zero.
\item\label{step4} Check if there is an element $g_{y'_w}\in G_{\geq 0}^{\delta}$ such that we have $x'_w(c)=\Ad(g_{y'_w})(y'_w)$. As in \ref{checkifing2dg2d} we will use Section \ref{Ad-action} to compute $\Ad(g_{y'_w})(y'_w)$ and determine a Gröbner basis to solve the resulting system. Note that the solution $g_{y'_w}$ will depend on the variables $c_{\al_i}$ that $x'_w(c)$ depends on.\label{checkconjugate}
\begin{enumerate}[leftmargin=*,label=Case \arabic*:]
\item If such an element $g_{y'_w}\in G_{\geq 0}^{\delta}$ with $x'_w(c)=\Ad(g_{y'_w})(y'_w)$ exists, we can focus our computations on $x'_w(c)$:
\begin{enumerate}[leftmargin=*,label=Case 1\alph*:]
\item If there is no $g'\in G\setminus G_{\geq 0}^{\delta}$ such that $\Ad(g')(y'_w)=y'_w,$ then $y'_w\in\fg_2^{\delta!}$ and so both $y_w\in\cN_{\fg}^{\btri}$ and $x\in\cN_{\fg}^{\btri}$.
\item There is $g'\in G\setminus G_{\geq 0}^{\delta}$ such that $\Ad(g')(y'_w)=y'_w$. Then $y'_w\notin \fg_2^{\delta!}$. As $x'_w(c)=\Ad(g_{y'_w})(y'_w)$ we have
\begin{align*}
\Ad(g'\inv{g_{y'_w}})(x'_w(c))&=\Ad(g')(y'_w)=y'_w
\end{align*}
and it follows that $\Ad(g_{y'_w}g'\inv{g_{y'_w}})(x'_w(c))=x'_w(c)$.\\
As before $g_{y'_w}g'\inv{g_{y'_w}}\notin G_{\geq 0}^{\delta},$ so $x'_w(c)\notin\fg_2^{\delta!}$.
\end{enumerate}
\item If there is no $g_{y'_w}\in G_{\geq 0}^{\delta}$ such that $x'_w(c)=\Ad(g_{y'_w})(y'_w)$, we first check whether Case 1b holds for $y'_w$. If it does, we check the same thing for $x'_w(c),$ i.e.\ if there is an element $g'\in G\setminus G_{\geq 0}^{\delta}$ such that $\Ad(g')(x'_w(c))=x'_w(c)$. \\
If Case 1b doesn't hold for $y'_w$, we will proceed as in \ref{checkcentralizer}. Note that $g'$ may depend on the variables in $x'_w(c),$ so we need to make sure that this solution holds for all possible values of the $c_{\al_i}$.\\
Sometimes we might find such a $g_{y'_w}\in G_{\geq 0}^{\delta}$ only for certain values of the $c_{\al_i}$. If this is the case, we will have to check everything in \ref{checkconjugate} for the values of the $\bar{c}_{\al_i}:=c_{\al_i}$ for which no $g_{y'_w}\in G_{\geq 0}^{\delta}$ as above exists. For these cases we need to check separately whether $x''_w(c)\in\fg_2^{\delta!}$ for the arising elements $x''_w(c):=[x(\bar{c}_{\al_1},\ldots,\bar{c}_{\al_r})]_{\fg_2^{\delta}}$ depending on the $\bar{c}_{\al_i}$.
\end{enumerate}
\item \label{checkcentralizer} For each $w''\in W$ first check whether $w''$ has weight $0,$ i.e.\ $w''\in W_{0}^{\delta}$. If this is not the case, let $g_{w''}\in Bn_{w''}B$ and note that $g_{w''}\notin G_{\geq 0}^{\delta}$ by Lemma \ref{checkg2d}. We want to check whether $g_{w''}$ centralises $y'_w,$ so as before we compute $\Ad(g_{w''})(y'_w)$ and solve the system $y'_w=\Ad(g_{w''})(y'_w)$ by using Gröbner bases. If this system has no solution, there is a possibility that $y'_w\in\fg_2^{\delta!}$ and we repeat \ref{checkcentralizer} for the next element ${w'''}\in W$.\\
If there is a solution, we take the elements $x''_w(c)$ from \ref{checkconjugate} for which we found no $g\in G_{\geq 0}^{\delta}$ such that $\Ad(g)(y'_w)=x''_w(c)$. We will check whether there is a solution for $x''_w(c)=\Ad(g_{w''})(x''_w(c))$ in each case. If we find a solution, it immediately follows that $x''_w(c)\notin \fg_2^{\delta!}$ and we move on to the next $w'''\in W$ in \ref{checkifing2dg2d}. If we have checked each $w'''\in W,$ it follows that this particular orbit is not contained in the piece.\\
If there exists $x''_w(c)$ as in \ref{step4}, Case 2, that is there is no element $g\in G_{\geq 0}^{\delta}$ with $x''_w(c)=\Ad(g)(y'_w),$ we need to check if there is a solution to $x''_w(c)=\Ad(g_{w'''})(x''_w(c))$ for each ${w'''}\in W$ for which we have already seen that $y'_w\neq\Ad(g_{w''})(y'_w)$. If so, then $x''_w(c)\notin\fg_2^{\delta!}$ and we move to the next $w'''\in W$ in \ref{checkifing2dg2d}. Otherwise we continue \ref{checkcentralizer} by replacing $y'_w$ with $x''_w(c)$.\\
If we do not find a solution for any element in $ W$ in \ref{checkcentralizer}, we have successfully proved that the centraliser of $x'_w(c)$ is contained in $G_{\geq 0}^{\delta}$ and therefore the orbit $\cO_x$ is contained in the piece. We can move on to check the next orbit. As we already know that certain orbits are contained in certain pieces, see Lemma \ref{diagcase}, we conjecture that in most cases $\cO_x$ will not be contained in $\cN_{\fg}^{\delta},$ see also Conjecture \ref{conjectureA}. This means we can run (and try to optimise) this algorithm keeping in mind that the most likely outcome is that there is an element $g\notin G_{\geq 0}^{\delta}$ centralising $[\Ad(n_wu')(x)]_{\fg_2^{\delta}}$.
\end{enumerate}

We can simplify the computations a bit:\\
If $n_{\tilde{w}}\in G_0^{\delta}$ for a fixed $\tilde{w}\in W$, check if $\Ad(n_{\tilde{w}})(z')=z'$ where $z'$ is as in \ref{cut2} above. If this is the case and we know that there exists a $g\in G\setminus G_{\geq 0}^{\delta}$ such that 
$\Ad(g)(z')\neq z'$ then for
$$h\in\{g\inv{n_{\tilde{w}}},gn_{\tilde{w}},n_{\tilde{w}}g,\inv{n_{\tilde{w}}}g,n_{\tilde{w}}g\inv{n_{\tilde{w}}},n_{\tilde{w}}gn_{\tilde{w}},\inv{n_{\tilde{w}}}g\inv{n_{\tilde{w}}},\inv{n_{\tilde{w}}}gn_{\tilde{w}}\}$$
it follows that $\Ad(h)(z')\neq z'$, as both $\Ad(n_{\tilde{w}})(z')=z'$ and $\Ad(n_{\tilde{w}}^{-1})(z')=z'$, as well as $\Ad(g)(z')\neq z'$.\\
Similarly, if $\tilde{u}\in U:=\prod_{\alpha\in\Phi^+}U_{\alpha}$ with $\Ad(\tilde{u})(z')=z'$ it is enough to check whether $\Ad(\tilde{u}g)(z')=z'$ or $\Ad(g\tilde{u})(z')=z'$. This means we may choose $\tilde{u}$ in the centralizer of $z'$ in such a way that $g$ depends on less variables $t_{\alpha}$: If $g=u'tn_wu$ as in Definition \ref{bruhatdecomp} where $u'=u_{\al_1}(c'_{\al_1})\cdots u_{\al_r}(c'_{\al_r})$ and $u=u_{\al_1}(c_{\al_1})\cdots u_{\al_r}(c_{\al_r})$ for $c_{\al_i},\,c'_{\al_i}\in\k$ then choose $\tilde{u}_1,\,\tilde{u}_2\in U$ such that $v':=\tilde{u}_1u'$ and $v:=\tilde{u}_2u$ depend on less variables $c_{\al_i},\,c'_{\al_i}\in\k$. It will now be easier to compute $\Ad(\tilde{u}_1gw\tilde{u}_2)(z')=z'$.
We summarise this section with the following pseudocode:
\section*{A pseudocode to compute the nilpotent pieces}
We summarise this section with the following pseudocode. The code is rather brief and only sketches the most important steps in the algorithm. It checks whether for a given weighted Dynkin diagram $\delta$ a nilpotent orbit with representative $x$ lies in the corresponding nilpotent piece.

\begin{algorithm}[h]
\caption{Computation of a nilpotent piece described by $\delta$}
\begin{algorithmic}[1]
\Procedure {Pieces}{$x,$ $S,$ $W,$ $\delta$}

INPUT:

$x$ - nilpotent orbit representative

$S$ - type of root system

$W$ - Weyl group

$\delta$ - weighted Dynkin diagram

OUTPUT: result whether $x$ is in $\cN_{\fg}^{\delta}$
\State find all elements $u\in U=\prod_{\al\in\Phi^+}U_{\al}$ s.t. $\Ad(u)(x)=x$ \textcolor{gray}{\textit{\small // we get $U$ from the root system}}
\State save the indices $i$ of $u_{\al_i}(c_i)$ in a list $Lu$ where $u_{\al_i}(c_i)$ is a factor in $u$ such that $c_i$ can be chosen arbitrarily and $\Ad(u_{\al_i}(c_i))(x)=x$
	
\State save all $w\in W$ with weight $0$ in a list $L$
\State compute a right transversal $T$ of all the elements of weight $0$ in $W$
\State $L2:=[~]$ \textcolor{gray}{\textit{\small //initiate list of elements in $W$ }}

\ForAll {$w\in T,$ $w\notin L2$}
	\State $y:=\Ad(n_wu)(x),$ $u$ not containing any factors $u_{\al_i}$, $i\in Lu$
	\If {$y\in \fg_{\geq 2}^{\delta}$}
	    \algstore{bkbreak}
		\end{algorithmic}
		\end{algorithm}
		\begin{algorithm}[h]
		\begin{algorithmic}[1]
		\algrestore{bkbreak}
		\State $z\leftarrow$ replace all variables $c_i$ in $y$ by 0 or 1 \textcolor{gray}{\textit{\small // in practice these two values are usually enough. We first check if this works for $0$, otherwise we use $1$. If both values do not work, we have to choose another value in the prime field or not replace the variable.}}
		\State $y'_w:=[y]_{\fg_2^{\delta}}$
		\State $z':=[z]_{\fg_2^{\delta}}$
		\State check if there exists $g\in G_{\geq 0}^{\delta}$ such that $\Ad(g)(z')=y'_w$
		
		\If {there exist values of the $c_{i}$ such that there exists no such $g$}
			\State $z_1\leftarrow $ values  of the $c_{i}$ such that there exists no such $g$
		\EndIf
		\ForAll {$w'\in W$}
			\If {$w'\notin L$}
				\If {$\Ad(g_{w'})(z')\neq z'$}
					\State go to the next $w'$
				\Else 
					\State check the same for all $z_1$
					 \If {$\Ad(g_{w'})(z_1)\neq z_1$}
						\State go to the next $w'$
					\Else 
						\State go to the next $w$ and save all $g.w$ for $g\in G_{\geq 0}$ in $L2$
					\EndIf
				\EndIf
			\EndIf
		\EndFor
	\Else
		\State go to the next $w$ and save all $\Ad(g)(w)$ for $g\in G_{\geq 0}$ in $L2$
	\EndIf
\EndFor
\If {for all $w\in T$ there exists $w'\in W\setminus L$ with $\Ad(g_{w'})(z')=z'$}
	 \State \Return ``$x$ not in piece''
\Else
	\State \Return ``$x$ in piece''
\EndIf
\EndProcedure
\end{algorithmic}
\end{algorithm}
\newpage
\section{Results}\label{results}
We fix the following notation for this section: Let $\Pi=\{\al_1,\ldots,\al_r\}$ be the set of simple roots, for some $r\in\N$. Denote the elements $e_{\al}$ spanning $\fg_{\al}$, $\al\in\Phi$, by $e_{1^{i_1},2^{i_2},\ldots,r^{i_r}}:=e_{i_1\al_1+i_2\al_2+\cdots+i_r\al_r}$, $i_j\in\Z$ and $i_1\al_1+i_2\al_2+\cdots+i_r\al_r\in\Phi$. If $i_j=0$ we will simply leave $j^{i_j}$ out.
\begin{rem}\label{rem_CP-pieces}
Note that by \cite[Theorem 7.3]{ClarkePremet} the CP-pieces agree with the nilpotent pieces if the nilpotent pieces form a partition of the nilpotent variety $\cN_{\fg}$.
\end{rem}
\subsection{\texorpdfstring{$G_2$}{}}
We have the following weighted Dynkin diagrams of type $G_2$ (\cite[13.1, p.401]{Carter85}).
\begin{center}
\begin{figure}[h]
\begin{tikzpicture}[scale=.5]
	\draw (-1,0) node {(1)};
    \draw[thick] (0 cm,0) circle (.2cm);
    \draw[thick] (1.5 cm,0) circle (.2cm);
    \draw (0 cm,-0.75) node  {\scriptsize{0}};
    \draw (1.5 cm,-0.75) node  {\scriptsize{0}};
    
    \draw[thick] (0.2 cm,0) -- +(1.1 cm,0);
    \draw[thick] (0 cm,0.2) -- +(1.5 cm,0);
    \draw[thick] (0 cm,-0.2) -- +(1.5 cm,0);
    \draw[thick] (0.85cm,0.2cm) --(0.5 cm,0)--(0.85cm, -0.2cm);
    
	\draw (4,0) node {(2)};    
    \draw[thick] (5 cm,0) circle (.2cm);
    \draw[thick] (6.5 cm,0) circle (.2cm);
    \draw (5 cm,-0.75) node  {\scriptsize{0}};
    \draw (6.5 cm,-0.75) node  {\scriptsize{1}};
    
    \draw[thick] (5.2 cm,0) -- +(1.1 cm,0);
    \draw[thick] (5 cm,0.2) -- +(1.5 cm,0);
    \draw[thick] (5 cm,-0.2) -- +(1.5 cm,0);
    \draw[thick] (5.85cm,0.2cm) --(5.5 cm,0)--(5.85cm, -0.2cm);
    
	\draw (9,0) node {(3)};    
    \draw[thick] (10 cm,0) circle (.2cm);
    \draw[thick] (11.5 cm,0) circle (.2cm);
    \draw (10 cm,-0.75) node  {\scriptsize{1}};
    \draw (11.5 cm,-0.75) node  {\scriptsize{0}};
    
    \draw[thick] (10.2 cm,0) -- +(1.1 cm,0);
    \draw[thick] (10 cm,0.2) -- +(1.5 cm,0);
    \draw[thick] (10 cm,-0.2) -- +(1.5 cm,0);
    \draw[thick] (10.85cm,0.2cm) --(10.5 cm,0)--(10.85cm, -0.2cm);
    
	\draw (-1,-2) node {(4)};    
    \draw[thick] (0 cm,-2) circle (.2cm);
    \draw[thick] (1.5 cm,-2) circle (.2cm);
    \draw (0 cm,-2.75) node  {\scriptsize{0}};
    \draw (1.5 cm,-2.75) node  {\scriptsize{2}};
    
    \draw[thick] (0.2 cm,-2) -- +(1.1 cm,0);
    \draw[thick] (0 cm,-1.8) -- +(1.5 cm,0);
    \draw[thick] (0 cm,-2.2) -- +(1.5 cm,0);
    \draw[thick] (0.85cm,-1.8cm) --(0.5 cm,-2)--(0.85cm, -2.2cm);
    
	\draw (4,-2) node {(5)};    
    \draw[thick] (5 cm,-2) circle (.2cm);
    \draw[thick] (6.5 cm,-2) circle (.2cm);
    \draw (5 cm,-2.75) node  {\scriptsize{2}};
    \draw (6.5 cm,-2.75) node  {\scriptsize{2}};
    
    \draw[thick] (5.2 cm,-2) -- +(1.1 cm,0);
    \draw[thick] (5 cm,-1.8) -- +(1.5 cm,0);
    \draw[thick] (5 cm,-2.2) -- +(1.5 cm,0);
    \draw[thick] (5.85cm,-1.8cm) --(5.5 cm,-2)--(5.85cm, -2.2cm);
\end{tikzpicture}
\small
\caption{The weighted Dynkin diagrams of type $G_2$}\label{wdg2}
 \normalsize
\end{figure}
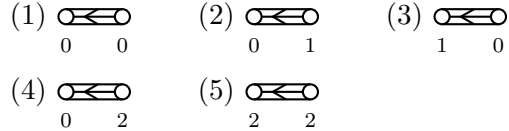
\end{center}
Furthermore, let the root system be given by 
$$\Phi:=\{\pm\alpha,\pm\beta,\pm(\alpha+\beta),\pm(2\alpha+\beta),\pm(3\alpha+\beta),\pm(3\alpha+2\beta)\}.$$
We begin with the case where char$(\k)=2$. By \cite{Stuhler}, we get the following orbit representatives, where we let $\fg_{\gamma}=\langle e_{\gamma}\rangle$ for $\gamma\in\Phi$.
\begin{multicols}{2}
\begin{enumerate}
\item $x_1:=0$,
\item $x_2:=e_{\alpha}+e_{\beta}$,
\item $x_3:=e_{\alpha}+e_{2\alpha+\beta}$,
\item $x_4:=e_{\alpha}$,
\item $x_5:=e_{\beta}$.
\end{enumerate}
\end{multicols}
Let $O_i$ be the orbit corresponding to the representative $x_i$, $i\in\{1,\ldots,5\}$ and $\cN_i$ be the piece corresponding to the weighted Dynkin diagram ($i$), $i\in\{1,\ldots,5\}$ as in figure \ref{wdg2}.
Then we get the following pieces:
\begin{multicols}{2}
\begin{enumerate}
\item $\cN_1:=\{0\}=O_1$,
\item $\cN_2:=O_5$,
\item $\cN_3:=O_4$,
\item $\cN_4:=O_3$,
\item $\cN_5:=O_2$.
\end{enumerate}
\end{multicols}
If char$(\k)=3$, the following orbit representatives are again given by \cite{Stuhler} using the same notation as above:
\begin{multicols}{2}
\begin{enumerate}
\item $x_1:=0$,
\item $x_2:=e_{\alpha}+e_{\beta}$,
\item $x_3:=e_{\beta}+e_{\alpha+\beta}$,
\item $x_4:=e_{\beta}+e_{2\alpha+\beta}$,
\item $x_5:=e_{\alpha}$,
\item $x_6:=e_{\beta}$.
\end{enumerate}
\end{multicols}
Then we get the following pieces:
\begin{multicols}{2}
\begin{enumerate}
\item $\cN_1:=\{0\}=O_1$,
\item $\cN_2:=O_6$,
\item $\cN_3:=O_3\cup O_5$,
\item $\cN_4:=O_4$,
\item $\cN_5:=O_2$.
\end{enumerate}
\end{multicols}
Note that in good characteristic, the representatives $x_3$ and $x_5$ are conjugate. We summarise the above results as follows:
\begin{thm}[Nilpotent pieces for $G_2$ in characteristic $2$ and $3$]
We use the same notation as above and let
\begin{multicols}{2}
\begin{enumerate}
\item $x_1:=0$,
\item $x_2:=e_{\alpha}+e_{\beta}$,
\item $x_3:=e_{\beta}+e_{\alpha+\beta}$,
\item $x_4:=e_{\beta}+e_{2\alpha+\beta}$,
\item $x_5:=e_{\alpha}$,
\item $x_6:=e_{\beta}$.
\end{enumerate}
\end{multicols} Then the pieces for $G_2$ with respect to the weighted Dynkin diagrams are given by 
\begin{multicols}{2}
\begin{enumerate}
\item $\cN_1:=\{0\}=O_{x_1}$,
\item $\cN_2:=O_{x_6}$,
\item $\cN_3:=O_{x_5}$,
\item $\cN_4:=O_{x_4}$,
\item $\cN_5:=O_{x_2}$
\end{enumerate}
\end{multicols}
if $\chara(k)=2$ and by
\begin{multicols}{2}
\begin{enumerate}
\item $\cN_1:=\{0\}=O_{x_1}$,
\item $\cN_2:=O_{x_6}$,
\item $\cN_3:=O_{x_3}\cup O_{x_5}$,
\item $\cN_4:=O_{x_4}$,
\item $\cN_5:=O_{x_2}$.
\end{enumerate}
\end{multicols}
if $\chara(k)=3$. In particular, the nilpotent pieces form a partition of the nilpotent variety and therefore agree with the CP-pieces by Remark \ref{rem_CP-pieces}.
\end{thm}
\subsection{\texorpdfstring{$F_4$}{}}
We proceed as above, by listing the weighted Dynkin diagrams for $F_4$ as well as the orbit representatives in good and bad characteristic. Note that the orbit representatives in good characteristic are the same as in characteristic $3$. For the orbit representatives see \cite[Table 22.1.4]{LiebeckSeitz} and \cite{Spaltenstein}. The simple roots are as given in the Dynkin diagram:\\
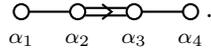
\begin{center}
\begin{tikzpicture}[scale=.5]
    \draw[thick] (1.5 cm,10) circle (.2cm);
    \draw[thick] (3 cm,10) circle (.2cm);
    \draw[thick] (4.5 cm,10) circle (.2cm);
    \draw[thick] (6 cm,10) circle (.2cm);
    \draw (1.5 cm,9.25) node  {\scriptsize{$\al_1$}};
    \draw (3 cm,9.25) node  {\scriptsize{$\al_2$}};
    \draw (4.5 cm,9.25) node  {\scriptsize{$\al_3$}};
    \draw (6 cm,9.25) node  {\scriptsize{$\al_4$}};
    
    \draw[thick] (1.7 cm,10) -- +(1.1 cm,0);
    \draw[thick] (3.2 cm,10.1) -- +(1.1 cm,0);
    \draw[thick] (3.2 cm,9.9) -- +(1.1 cm,0);
    \draw[thick] (3.65cm, 10.2 cm) --(3.9 cm,10)--(3.65cm, 9.8cm);
    \draw[thick] (4.7 cm,10) -- +(1.1 cm,0);
    \draw (6.5,9.9) node {.};
   \end{tikzpicture}\\
   
\end{center}
The nilpotent orbit representatives are as in the following table. As they are the same in good characteristic and for characteristic 3, these cases are not distinguished. In characteristic two we get additional orbits, which are denoted by $x_{i,2}$ if they are in the same orbit as $x_i$ in good characteristic. All orbit representatives in good characteristic (or characteristic $3$) are also orbit representatives in characteristic $2$.
\renewcommand{\arraystretch}{2}
\begin{longtable}{|r|c|c|r@{:= }>{\RaggedRight}p{3cm}|}
\hline
&  \multirow{2}*{Label} &  \multirow{2}*{Weighted Dynkin Diagram} & \multicolumn{2}{|c|}{Nilpotent orbit}\\[-0.4cm]
& & &\multicolumn{2}{|c|}{representative}\\
\hline
(i) & $1$& 
\begin{tikzpicture}[scale=.4,baseline=(y)]
    \draw[thick] (1.5 cm,0) circle (.2cm);
    \draw[thick] (3 cm,0) circle (.2cm);
    \draw[thick] (4.5 cm,0) circle (.2cm);
    \draw[thick] (6 cm,0) circle (.2cm);
    \draw (1.5 cm,-0.75) node  {\scriptsize{0}};
    \draw (3 cm,-0.75) node  {\scriptsize{0}};
    \draw (4.5 cm,-0.75) node  {\scriptsize{0}};
    \draw (6 cm,-0.75) node  {\scriptsize{0}};
    
    \draw[thick] (1.7 cm,0) -- +(1.1 cm,0);
    \draw[thick] (3.2 cm,0.1) -- +(1.1 cm,0);
    \draw[thick] (3.2 cm,-0.1) -- +(1.1 cm,0);
    \draw[thick] (3.65cm, 0.2 cm) --(3.9 cm,0)--(3.65cm, -0.2cm);
    \draw[thick] (4.7 cm,0) -- +(1.1 cm,0);
    \draw (2,-0.5) node (y)  {};
    \end{tikzpicture}
      & $x_1$ & $0$ \\
\hline
(ii) & $A_1$ &
\begin{tikzpicture}[scale=.4,baseline=(y)]
    \draw[thick] (1.5 cm,0) circle (.2cm);
    \draw[thick] (3 cm,0) circle (.2cm);
    \draw[thick] (4.5 cm,0) circle (.2cm);
    \draw[thick] (6 cm,0) circle (.2cm);
    \draw (1.5 cm,-0.75) node  {\scriptsize{1}};
    \draw (3 cm,-0.75) node  {\scriptsize{0}};
    \draw (4.5 cm,-0.75) node  {\scriptsize{0}};
    \draw (6 cm,-0.75) node  {\scriptsize{0}};
    
    \draw[thick] (1.7 cm,0) -- +(1.1 cm,0);
    \draw[thick] (3.2 cm,0.1) -- +(1.1 cm,0);
    \draw[thick] (3.2 cm,-0.1) -- +(1.1 cm,0);
    \draw[thick] (3.65cm, 0.2 cm) --(3.9 cm,0)--(3.65cm, -0.2cm);
    \draw[thick] (4.7 cm,0) -- +(1.1 cm,0);
    \draw (2,-0.5) node (y)  {};
    \end{tikzpicture}
      & $x_2$ & $e_{1^2,2^3,3^4,4^2}$ \\
\hline
 \multirow{2}*{(iii)} & $\tilde{A}_1$ &
 \multirow{2}*{
 \begin{tikzpicture}[scale=.4,baseline=(y)]
    \draw[thick] (1.5 cm,0) circle (.2cm);
    \draw[thick] (3 cm,0) circle (.2cm);
    \draw[thick] (4.5 cm,0) circle (.2cm);
    \draw[thick] (6 cm,0) circle (.2cm);
    \draw (1.5 cm,-0.75) node  {\scriptsize{0}};
    \draw (3 cm,-0.75) node  {\scriptsize{0}};
    \draw (4.5 cm,-0.75) node  {\scriptsize{0}};
    \draw (6 cm,-0.75) node  {\scriptsize{1}};
    \draw[thick] (1.7 cm,0) -- +(1.1 cm,0);
    \draw[thick] (3.2 cm,0.1) -- +(1.1 cm,0);
    \draw[thick] (3.2 cm,-0.1) -- +(1.1 cm,0);
    \draw[thick] (3.65cm, 0.2 cm) --(3.9 cm,0)--(3.65cm, -0.2cm);
    \draw[thick] (4.7 cm,0) -- +(1.1 cm,0);
    \draw (2,-0.5) node (y)  {};
    \end{tikzpicture}
    }
     & $x_{3}$ & $e_{1,2^2,3^3,4^2}$ \\
      & $(\tilde{A}_1)_2$ & & $x_{3,2}$ & $e_{1,2^2,3^3,4^2}+e_{1^2,2^3,3^4,4^2}$, $p=2$\\
      
\hline
(iv) & $A_1\tilde{A}_1$ &
\begin{tikzpicture}[scale=.4,baseline=(y)]
    \draw[thick] (1.5 cm,0) circle (.2cm);
    \draw[thick] (3 cm,0) circle (.2cm);
    \draw[thick] (4.5 cm,0) circle (.2cm);
    \draw[thick] (6 cm,0) circle (.2cm);
    \draw (1.5 cm,-0.75) node  {\scriptsize{0}};
    \draw (3 cm,-0.75) node  {\scriptsize{1}};
    \draw (4.5 cm,-0.75) node  {\scriptsize{0}};
    \draw (6 cm,-0.75) node  {\scriptsize{0}};
    
    \draw[thick] (1.7 cm,0) -- +(1.1 cm,0);
    \draw[thick] (3.2 cm,0.1) -- +(1.1 cm,0);
    \draw[thick] (3.2 cm,-0.1) -- +(1.1 cm,0);
    \draw[thick] (3.65cm, 0.2 cm) --(3.9 cm,0)--(3.65cm, -0.2cm);
    \draw[thick] (4.7 cm,0) -- +(1.1 cm,0);
    \draw (2,-0.5) node (y)  {};
    \end{tikzpicture}
      & $x_4$ & $e_{1,2^2,3^3,4}+e_{1,2^2,3^2,4^2}$ \\
\hline
(v) & $A_2$ &
\begin{tikzpicture}[scale=.4,baseline=(y)]
    \draw[thick] (1.5 cm,0) circle (.2cm);
    \draw[thick] (3 cm,0) circle (.2cm);
    \draw[thick] (4.5 cm,0) circle (.2cm);
    \draw[thick] (6 cm,0) circle (.2cm);
    \draw (1.5 cm,-0.75) node  {\scriptsize{2}};
    \draw (3 cm,-0.75) node  {\scriptsize{0}};
    \draw (4.5 cm,-0.75) node  {\scriptsize{0}};
    \draw (6 cm,-0.75) node  {\scriptsize{0}};
    
    \draw[thick] (1.7 cm,0) -- +(1.1 cm,0);
    \draw[thick] (3.2 cm,0.1) -- +(1.1 cm,0);
    \draw[thick] (3.2 cm,-0.1) -- +(1.1 cm,0);
    \draw[thick] (3.65cm, 0.2 cm) --(3.9 cm,0)--(3.65cm, -0.2cm);
    \draw[thick] (4.7 cm,0) -- +(1.1 cm,0);
    \draw (2,-0.5) node (y)  {};
    \end{tikzpicture}
      & $x_5$ & $e_{1,2^2,3^2}+e_{1,2,3^2,4^2}$ \\
\hline
 \multirow{2}*{(vi)} & $\tilde{A}_2$ & \multirow{2}*{
\begin{tikzpicture}[scale=.4,baseline=(y)]
    \draw[thick] (1.5 cm,0) circle (.2cm);
    \draw[thick] (3 cm,0) circle (.2cm);
    \draw[thick] (4.5 cm,0) circle (.2cm);
    \draw[thick] (6 cm,0) circle (.2cm);
    \draw (1.5 cm,-0.75) node  {\scriptsize{0}};
    \draw (3 cm,-0.75) node  {\scriptsize{0}};
    \draw (4.5 cm,-0.75) node  {\scriptsize{0}};
    \draw (6 cm,-0.75) node  {\scriptsize{2}};
    \draw[thick] (1.7 cm,0) -- +(1.1 cm,0);
    \draw[thick] (3.2 cm,0.1) -- +(1.1 cm,0);
    \draw[thick] (3.2 cm,-0.1) -- +(1.1 cm,0);
    \draw[thick] (3.65cm, 0.2 cm) --(3.9 cm,0)--(3.65cm, -0.2cm);
    \draw[thick] (4.7 cm,0) -- +(1.1 cm,0);
    \draw (2,-0.5) node (y)  {};
    \end{tikzpicture}
    }
      & $x_{6}$ & $e_{1,2,3,4}+e_{2,3^2,4}$ \\
     &$(\tilde{A}_2)_2$ & & $x_{6,2}$ & $e_{1,2,3,4}+e_{2,3^2,4}+e_{1^2,2^3,3^4,4^2}$, $p=2$ \\
\hline
(vii) & $A_2\tilde{A}_1$ &
\begin{tikzpicture}[scale=.4,baseline=(y)]
    \draw[thick] (1.5 cm,0) circle (.2cm);
    \draw[thick] (3 cm,0) circle (.2cm);
    \draw[thick] (4.5 cm,0) circle (.2cm);
    \draw[thick] (6 cm,0) circle (.2cm);
    \draw (1.5 cm,-0.75) node  {\scriptsize{0}};
    \draw (3 cm,-0.75) node  {\scriptsize{0}};
    \draw (4.5 cm,-0.75) node  {\scriptsize{1}};
    \draw (6 cm,-0.75) node  {\scriptsize{0}};
    
    \draw[thick] (1.7 cm,0) -- +(1.1 cm,0);
    \draw[thick] (3.2 cm,0.1) -- +(1.1 cm,0);
    \draw[thick] (3.2 cm,-0.1) -- +(1.1 cm,0);
    \draw[thick] (3.65cm, 0.2 cm) --(3.9 cm,0)--(3.65cm, -0.2cm);
    \draw[thick] (4.7 cm,0) -- +(1.1 cm,0);
    \draw (2,-0.5) node (y)  {};
    \end{tikzpicture}
      & $x_7$ & $e_{1,2^2,3^2}+e_{1,2,3^2,4}+e_{2,3^2,4^2}$ \\
\hline
 \multirow{2}*{(viii)} & $B_2$ & \multirow{2}*{
\begin{tikzpicture}[scale=.4,baseline=(y)]
    \draw[thick] (1.5 cm,0) circle (.2cm);
    \draw[thick] (3 cm,0) circle (.2cm);
    \draw[thick] (4.5 cm,0) circle (.2cm);
    \draw[thick] (6 cm,0) circle (.2cm);
    \draw (1.5 cm,-0.75) node  {\scriptsize{2}};
    \draw (3 cm,-0.75) node  {\scriptsize{0}};
    \draw (4.5 cm,-0.75) node  {\scriptsize{0}};
    \draw (6 cm,-0.75) node  {\scriptsize{1}};
    \draw[thick] (1.7 cm,0) -- +(1.1 cm,0);
    \draw[thick] (3.2 cm,0.1) -- +(1.1 cm,0);
    \draw[thick] (3.2 cm,-0.1) -- +(1.1 cm,0);
    \draw[thick] (3.65cm, 0.2 cm) --(3.9 cm,0)--(3.65cm, -0.2cm);
    \draw[thick] (4.7 cm,0) -- +(1.1 cm,0);
    \draw (2,-0.5) node (y)  {};
    \end{tikzpicture}
    }
      & $x_{8}$ & $e_{1,2,3}+e_{2,3^2,4^2}$ \\
     &$(B_2)_2$  & & $x_{8,2}$ & $e_{1,2}+e_{1,2,3^2}+e_{2,3^2,4^2}$, $p=2$\\
\hline
 \multirow{2}*{(ix)} & $\tilde{A}_2A_1$ &\multirow{2}*{
\begin{tikzpicture}[scale=.4,baseline=(y)]
    \draw[thick] (1.5 cm,0) circle (.2cm);
    \draw[thick] (3 cm,0) circle (.2cm);
    \draw[thick] (4.5 cm,0) circle (.2cm);
    \draw[thick] (6 cm,0) circle (.2cm);
    \draw (1.5 cm,-0.75) node  {\scriptsize{0}};
    \draw (3 cm,-0.75) node  {\scriptsize{1}};
    \draw (4.5 cm,-0.75) node  {\scriptsize{0}};
    \draw (6 cm,-0.75) node  {\scriptsize{1}};
    \draw[thick] (1.7 cm,0) -- +(1.1 cm,0);
    \draw[thick] (3.2 cm,0.1) -- +(1.1 cm,0);
    \draw[thick] (3.2 cm,-0.1) -- +(1.1 cm,0);
    \draw[thick] (3.65cm, 0.2 cm) --(3.9 cm,0)--(3.65cm, -0.2cm);
    \draw[thick] (4.7 cm,0) -- +(1.1 cm,0);
    \draw (2,-0.5) node (y)  {};
    \end{tikzpicture}
    }
    & $x_{9}$ & $e_{1,2,3}+e_{2,3^2,4}+e_{1,2^2,3^2,4^2}$ \\
      & $(\tilde{A}_2A_1)_2$  & & $x_{9,2}$ & $e_{1,2,3}+e_{2,3^2,4}+e_{1,2^2,3^2,4^2}+e_{1,2^2,3^2}$, $p=2$ \\
      
\hline
 \multirow{2}*{(x)} &  $C_3(a_1)$ & \multirow{2}*{
\begin{tikzpicture}[scale=.4,baseline=(y)]
    \draw[thick] (1.5 cm,0) circle (.2cm);
    \draw[thick] (3 cm,0) circle (.2cm);
    \draw[thick] (4.5 cm,0) circle (.2cm);
    \draw[thick] (6 cm,0) circle (.2cm);
    \draw (1.5 cm,-0.75) node  {\scriptsize{1}};
    \draw (3 cm,-0.75) node  {\scriptsize{0}};
    \draw (4.5 cm,-0.75) node  {\scriptsize{1}};
    \draw (6 cm,-0.75) node  {\scriptsize{0}};
    \draw[thick] (1.7 cm,0) -- +(1.1 cm,0);
    \draw[thick] (3.2 cm,0.1) -- +(1.1 cm,0);
    \draw[thick] (3.2 cm,-0.1) -- +(1.1 cm,0);
    \draw[thick] (3.65cm, 0.2 cm) --(3.9 cm,0)--(3.65cm, -0.2cm);
    \draw[thick] (4.7 cm,0) -- +(1.1 cm,0);
    \draw (2,-0.5) node (y)  {};
    \end{tikzpicture}
    }
      & $x_{10}$ & $e_{1,2,3}+e_{2,3^2,4}+e_{2,3^2,4^2}$ \\
     & $(C_3(a_1))_2$ & & $x_{10,2}$ & $e_{1,2}+e_{2,3^2,4}+e_{1,2,3^2,4^2}+e_{1,2^2,3^4,4^2}$, $p=2$ \\
\hline
(xi) & $F_4(a_3)$ &
\begin{tikzpicture}[scale=.4,baseline=(y)]
    \draw[thick] (1.5 cm,0) circle (.2cm);
    \draw[thick] (3 cm,0) circle (.2cm);
    \draw[thick] (4.5 cm,0) circle (.2cm);
    \draw[thick] (6 cm,0) circle (.2cm);
    \draw (1.5 cm,-0.75) node  {\scriptsize{0}};
    \draw (3 cm,-0.75) node  {\scriptsize{2}};
    \draw (4.5 cm,-0.75) node  {\scriptsize{0}};
    \draw (6 cm,-0.75) node  {\scriptsize{0}};
    
    \draw[thick] (1.7 cm,0) -- +(1.1 cm,0);
    \draw[thick] (3.2 cm,0.1) -- +(1.1 cm,0);
    \draw[thick] (3.2 cm,-0.1) -- +(1.1 cm,0);
    \draw[thick] (3.65cm, 0.2 cm) --(3.9 cm,0)--(3.65cm, -0.2cm);
    \draw[thick] (4.7 cm,0) -- +(1.1 cm,0);
    \draw (2,-0.5) node (y)  {};
    \end{tikzpicture}
      & $x_{11}$ & $e_{1,2,3}+e_{1,2,3^2}+e_{2,3,4}+e_{2^2,3^2,4^2}$ \\
\hline
(xii) & $B_3$ &
\begin{tikzpicture}[scale=.4,baseline=(y)]
    \draw[thick] (1.5 cm,0) circle (.2cm);
    \draw[thick] (3 cm,0) circle (.2cm);
    \draw[thick] (4.5 cm,0) circle (.2cm);
    \draw[thick] (6 cm,0) circle (.2cm);
    \draw (1.5 cm,-0.75) node  {\scriptsize{2}};
    \draw (3 cm,-0.75) node  {\scriptsize{2}};
    \draw (4.5 cm,-0.75) node  {\scriptsize{0}};
    \draw (6 cm,-0.75) node  {\scriptsize{0}};
    
    \draw[thick] (1.7 cm,0) -- +(1.1 cm,0);
    \draw[thick] (3.2 cm,0.1) -- +(1.1 cm,0);
    \draw[thick] (3.2 cm,-0.1) -- +(1.1 cm,0);
    \draw[thick] (3.65cm, 0.2 cm) --(3.9 cm,0)--(3.65cm, -0.2cm);
    \draw[thick] (4.7 cm,0) -- +(1.1 cm,0);
    \draw (2,-0.5) node (y)  {};
    \end{tikzpicture}
      & $x_{12}$ & $e_{1 }+e_{2,3}+e_{2,3,4}+e_{2,3^2,4^2}$ \\
      \hline
\multirow{2}*{(xiii)} & $C_3$ & \multirow{2}*{
\begin{tikzpicture}[scale=.4,baseline=(y)]
    \draw[thick] (1.5 cm,0) circle (.2cm);
    \draw[thick] (3 cm,0) circle (.2cm);
    \draw[thick] (4.5 cm,0) circle (.2cm);
    \draw[thick] (6 cm,0) circle (.2cm);
    \draw (1.5 cm,-0.75) node  {\scriptsize{1}};
    \draw (3 cm,-0.75) node  {\scriptsize{0}};
    \draw (4.5 cm,-0.75) node  {\scriptsize{1}};
    \draw (6 cm,-0.75) node  {\scriptsize{2}};
    \draw[thick] (1.7 cm,0) -- +(1.1 cm,0);
    \draw[thick] (3.2 cm,0.1) -- +(1.1 cm,0);
    \draw[thick] (3.2 cm,-0.1) -- +(1.1 cm,0);
    \draw[thick] (3.65cm, 0.2 cm) --(3.9 cm,0)--(3.65cm, -0.2cm);
    \draw[thick] (4.7 cm,0) -- +(1.1 cm,0);
    \draw (2,-0.5) node (y)  {};
    \end{tikzpicture}
    }
	 & $x_{13}$ & $e_{4}+e_{1,2,3}+e_{2,3^2}$  \\    
    & $(C_3)_2$  & & $x_{13,2}$ & $e_{4}+e_{1,2,3}+e_{2,3^2}+e_{1,2^2,3^2,4^2}$, $p=2$\\ \hline
(xiv) & $F_4(a_2)$ &
\begin{tikzpicture}[scale=.4,baseline=(y)]
    \draw[thick] (1.5 cm,0) circle (.2cm);
    \draw[thick] (3 cm,0) circle (.2cm);
    \draw[thick] (4.5 cm,0) circle (.2cm);
    \draw[thick] (6 cm,0) circle (.2cm);
    \draw (1.5 cm,-0.75) node  {\scriptsize{0}};
    \draw (3 cm,-0.75) node  {\scriptsize{2}};
    \draw (4.5 cm,-0.75) node  {\scriptsize{0}};
    \draw (6 cm,-0.75) node  {\scriptsize{2}};
    
    \draw[thick] (1.7 cm,0) -- +(1.1 cm,0);
    \draw[thick] (3.2 cm,0.1) -- +(1.1 cm,0);
    \draw[thick] (3.2 cm,-0.1) -- +(1.1 cm,0);
    \draw[thick] (3.65cm, 0.2 cm) --(3.9 cm,0)--(3.65cm, -0.2cm);
    \draw[thick] (4.7 cm,0) -- +(1.1 cm,0);
    \draw (2,-0.5) node (y)  {};
    \end{tikzpicture}
      & $x_{14}$ & $e_{1,2}+e_{2,3}+e_{3,4}+e_{1,2,3^2}$ \\
\hline
(xv) & $F_4(a_1)$ &
\begin{tikzpicture}[scale=.4,baseline=(y)]
    \draw[thick] (1.5 cm,0) circle (.2cm);
    \draw[thick] (3 cm,0) circle (.2cm);
    \draw[thick] (4.5 cm,0) circle (.2cm);
    \draw[thick] (6 cm,0) circle (.2cm);
    \draw (1.5 cm,-0.75) node  {\scriptsize{2}};
    \draw (3 cm,-0.75) node  {\scriptsize{2}};
    \draw (4.5 cm,-0.75) node  {\scriptsize{0}};
    \draw (6 cm,-0.75) node  {\scriptsize{2}};
    
    \draw[thick] (1.7 cm,0) -- +(1.1 cm,0);
    \draw[thick] (3.2 cm,0.1) -- +(1.1 cm,0);
    \draw[thick] (3.2 cm,-0.1) -- +(1.1 cm,0);
    \draw[thick] (3.65cm, 0.2 cm) --(3.9 cm,0)--(3.65cm, -0.2cm);
    \draw[thick] (4.7 cm,0) -- +(1.1 cm,0);
    \draw (2,-0.5) node (y)  {};
    \end{tikzpicture}
      & $x_{15}$ & $e_{1}+e_{4}+e_{2,3}+e_{2,3^2}$ \\
\hline
(xvi) & $F_4$ &
\begin{tikzpicture}[scale=.4,baseline=(y)]
    \draw[thick] (1.5 cm,0) circle (.2cm);
    \draw[thick] (3 cm,0) circle (.2cm);
    \draw[thick] (4.5 cm,0) circle (.2cm);
    \draw[thick] (6 cm,0) circle (.2cm);
    \draw (1.5 cm,-0.75) node  {\scriptsize{2}};
    \draw (3 cm,-0.75) node  {\scriptsize{2}};
    \draw (4.5 cm,-0.75) node  {\scriptsize{2}};
    \draw (6 cm,-0.75) node  {\scriptsize{2}};
    
    \draw[thick] (1.7 cm,0) -- +(1.1 cm,0);
    \draw[thick] (3.2 cm,0.1) -- +(1.1 cm,0);
    \draw[thick] (3.2 cm,-0.1) -- +(1.1 cm,0);
    \draw[thick] (3.65cm, 0.2 cm) --(3.9 cm,0)--(3.65cm, -0.2cm);
    \draw[thick] (4.7 cm,0) -- +(1.1 cm,0);
    \draw (2,-0.5) node (y)  {};
    \end{tikzpicture}
      & $x_{16}$ & $e_{1}+e_{2}+e_{3}+e_{4}$ \\
\hline
\end{longtable}
Applying the programme described in section \ref{computation} results in the following pieces in characteristic $2$ (with the same notation as for $G_2$):
\begin{multicols}{2}
\begin{enumerate}
\item $\cN_1=O_1$,
\item $\cN_2=O_2$,
\item $\cN_3=O_{3}\cup O_{3,2}$,
\item $\cN_4=O_4$,
\item $\cN_5=O_5$,
\item $\cN_6=O_{6}\cup O_{6,2}$,
\item $\cN_7=O_{7}$,
\item $\cN_8=O_{8}\cup O_{8,2}$,
\item $\cN_9=O_{9}\cup O_{9,2}$,
\item $\cN_{10}=O_{10}\cup O_{10,2}$,
\item $\cN_{11}=O_{11}$,
\item $\cN_{12}=O_{12}$,
\item $\cN_{13}=O_{13}\cup O_{13,2}$,
\item $\cN_{14}=O_{14}$,
\item $\cN_{15}=O_{15}$,
\item $\cN_{16}=O_{16}$,
\end{enumerate}
\end{multicols}
\noindent where $\cN_{16}$ follows by Proposition \ref{regpiece}.\\
Note that each of the pairs of orbit representatives $x_{3}$ and $x_{3,2}$, as well as $x_{6}$ and $x_{6,2}$, $x_{8}$ and $x_{8,2}$, $x_{9}$ and $x_{9,2}$, $x_{10}$ and $x_{10,2}$ and $x_{13}$ and $x_{13,2}$ are conjugate to each other in good characteristic. In particular, the pieces contain the same orbits (or the orbits they split into) as in good characteristic. We expect this to be a pattern that should hopefully hold for the other exceptional groups as well.
In characteristic 3 we use the same orbit representatives (i.e. mapping the coefficients of the linear combinations of basis elements into $\k$).
Applying the programme results in the same pieces as in characteristic 0, i.e. we have $\cN_i=O_i$ for all $i=1,\ldots, 16$ with the same notation as above.
\begin{thm}[Nilpotent pieces for $F_4$ in characteristic $2$ and $3$]
We use the same notation as above. Then the pieces for $F_4$ in characteristic $2$ are given by 
$\cN_i=\cO_i\cup\cO_{i,2}$ for $i\in\{3,6,8,9,10,13\}$  and $\cN_i=\cO_i$ otherwise.\\
In characteristic $3$ we get  $\cN_i=\cO_i$ for all $1\leq i\leq 16$. In particular, the nilpotent pieces form a partition of the nilpotent variety and therefore agree with the CP-pieces by Remark \ref{rem_CP-pieces}.
\end{thm}
\subsection{\texorpdfstring{$E_6$}{}}
If the root system is of type $E_6$, we get the same amount of orbits in good and bad characteristic. In fact we can choose the \enquote{same} orbit representatives for each characteristic, where the coefficients of the $e_{\al}$ are either $0$ or $1$. We list them in the table below, together with the weighted Dynkin diagrams describing the orbits in good characteristic. As the root system is getting bigger, we will now denote the elements $e_{\al}$ spanning $\fg_{\al}$, $\al\in\Phi$, by $e_{1^{i_1},2^{i_2},\ldots,6^{i_6}}:=e_{i_1\al_1+i_2\al_2+\cdots+i_6\al_6}$. If $i_j=0$ we will simply leave $j^{i_j}$ out. Here, the roots $\{\al_1,\al_2,\ldots,\al_6\}$ are the simple roots of $\Phi$ as denoted in the Dynkin diagram:
\vspace*{-0.2cm}
\begin{center}
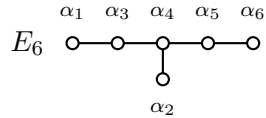
\begin{figure}[h]
\begin{tikzpicture}[scale=.4]
\draw (0,0) node {$E_6$};
    \draw[thick] (1.5 cm,0) circle (.2cm);
    \draw[thick] (3 cm,0) circle (.2cm);
    \draw[thick] (4.5 cm,0) circle (.2cm);
    \draw[thick] (6 cm,0) circle (.2cm);
    \draw[thick] (7.5 cm,0) circle (.2cm);
    \draw[thick] (4.5 cm,-1.2) circle (.2cm);
    
    \draw[thick] (1.7 cm,0) -- +(1.1 cm,0);
    \draw[thick] (3.2 cm,0) -- +(1.1 cm,0);
    \draw[thick] (4.7 cm,0) -- +(1.1 cm,0);
    \draw[thick] (6.2 cm,0) -- +(1.1 cm,0);
    \draw[thick] (4.5 cm,-0.2) -- +(0,-0.8cm);
    
    \draw (1.5 cm,1) node  {\scriptsize{$\al_1$}};
    \draw (4.5 cm,-2.2) node  {\scriptsize{$\al_2$}};
    \draw (3 cm,1) node  {\scriptsize{$\al_3$}};
    \draw (4.5 cm,1) node  {\scriptsize{$\al_4$}};
    \draw (6 cm,1) node  {\scriptsize{$\al_5$}};
    \draw (7.5 cm,1) node  {\scriptsize{$\al_6$}};
    \end{tikzpicture}
\small
\vspace*{-0.2cm}
\caption{The Dynkin diagram of type $E_6$}
 \normalsize
\end{figure}
\end{center}
\begin{center}
\begin{longtable}{|r|c|c|r@{:= }>{\RaggedRight}p{3cm}|}
\hline
&  \multirow{2}*{Label} &  \multirow{2}*{Weighted Dynkin Diagram} & \multicolumn{2}{|c|}{Nilpotent orbit}\\[-0.4cm]
& & &\multicolumn{2}{|c|}{representative}\\
\hline
(i) & $1$ &
\begin{tikzpicture}[scale=.3, baseline=(y)]
    \draw[thick] (1.5 cm,0) circle (.2cm);
    \draw[thick] (3 cm,0) circle (.2cm);
    \draw[thick] (4.5 cm,0) circle (.2cm);
    \draw[thick] (6 cm,0) circle (.2cm);
    \draw[thick] (7.5 cm,0) circle (.2cm);
    \draw[thick] (4.5 cm,-1.2) circle (.2cm);
    
    \draw[thick] (1.7 cm,0) -- +(1.1 cm,0);
    \draw[thick] (3.2 cm,0) -- +(1.1 cm,0);
    \draw[thick] (4.7 cm,0) -- +(1.1 cm,0);
    \draw[thick] (6.2 cm,0) -- +(1.1 cm,0);
    \draw[thick] (4.5 cm,-0.2) -- +(0,-0.8cm);
    
    \draw (1.5 cm,1) node  {\scriptsize{$0$}};
    \draw (4.5 cm,-2.2) node  {\scriptsize{$0$}};
    \draw (3 cm,1) node  {\scriptsize{$0$}};
    \draw (4.5 cm,1) node  {\scriptsize{$0$}};
    \draw (6 cm,1) node  {\scriptsize{$0$}};
    \draw (7.5 cm,1) node  {\scriptsize{$0$}};
    
    \draw (0,-1) node (y)  {};
    \end{tikzpicture}
    & $x_1$ & $0$ \\
\hline
(ii) & $A_1$ &
\begin{tikzpicture}[scale=.3, baseline=(y)]
    \draw[thick] (1.5 cm,0) circle (.2cm);
    \draw[thick] (3 cm,0) circle (.2cm);
    \draw[thick] (4.5 cm,0) circle (.2cm);
    \draw[thick] (6 cm,0) circle (.2cm);
    \draw[thick] (7.5 cm,0) circle (.2cm);
    \draw[thick] (4.5 cm,-1.2) circle (.2cm);
    
    \draw[thick] (1.7 cm,0) -- +(1.1 cm,0);
    \draw[thick] (3.2 cm,0) -- +(1.1 cm,0);
    \draw[thick] (4.7 cm,0) -- +(1.1 cm,0);
    \draw[thick] (6.2 cm,0) -- +(1.1 cm,0);
    \draw[thick] (4.5 cm,-0.2) -- +(0,-0.8cm);
    
    \draw (1.5 cm,1) node  {\scriptsize{$0$}};
    \draw (4.5 cm,-2.2) node  {\scriptsize{$1$}};
    \draw (3 cm,1) node  {\scriptsize{$0$}};
    \draw (4.5 cm,1) node  {\scriptsize{$0$}};
    \draw (6 cm,1) node  {\scriptsize{$0$}};
    \draw (7.5 cm,1) node  {\scriptsize{$0$}};
    
    \draw (0,-1) node (y)  {};
    \end{tikzpicture}
    & $x_2$&$e_{1,2^2,3^2,4^3,5^2,6}$
\\
\hline
(iii) & $A_1^2$ &
\begin{tikzpicture}[scale=.3, baseline=(y)]
    \draw[thick] (1.5 cm,0) circle (.2cm);
    \draw[thick] (3 cm,0) circle (.2cm);
    \draw[thick] (4.5 cm,0) circle (.2cm);
    \draw[thick] (6 cm,0) circle (.2cm);
    \draw[thick] (7.5 cm,0) circle (.2cm);
    \draw[thick] (4.5 cm,-1.2) circle (.2cm);
    
    \draw[thick] (1.7 cm,0) -- +(1.1 cm,0);
    \draw[thick] (3.2 cm,0) -- +(1.1 cm,0);
    \draw[thick] (4.7 cm,0) -- +(1.1 cm,0);
    \draw[thick] (6.2 cm,0) -- +(1.1 cm,0);
    \draw[thick] (4.5 cm,-0.2) -- +(0,-0.8cm);
    
    \draw (1.5 cm,1) node  {\scriptsize{$1$}};
    \draw (4.5 cm,-2.2) node  {\scriptsize{$0$}};
    \draw (3 cm,1) node  {\scriptsize{$0$}};
    \draw (4.5 cm,1) node  {\scriptsize{$0$}};
    \draw (6 cm,1) node  {\scriptsize{$0$}};
    \draw (7.5 cm,1) node  {\scriptsize{$1$}};
    
    \draw (0,-1) node (y)  {};
    \end{tikzpicture}
    & $x_3$&$e_{1,2,3^2,4^2,5,6}+e_{1,2,3,4^2,5^2,6}$
\\
\hline
(iv) & $A_2$ &
\begin{tikzpicture}[scale=.3, baseline=(y)]
    \draw[thick] (1.5 cm,0) circle (.2cm);
    \draw[thick] (3 cm,0) circle (.2cm);
    \draw[thick] (4.5 cm,0) circle (.2cm);
    \draw[thick] (6 cm,0) circle (.2cm);
    \draw[thick] (7.5 cm,0) circle (.2cm);
    \draw[thick] (4.5 cm,-1.2) circle (.2cm);
    
    \draw[thick] (1.7 cm,0) -- +(1.1 cm,0);
    \draw[thick] (3.2 cm,0) -- +(1.1 cm,0);
    \draw[thick] (4.7 cm,0) -- +(1.1 cm,0);
    \draw[thick] (6.2 cm,0) -- +(1.1 cm,0);
    \draw[thick] (4.5 cm,-0.2) -- +(0,-0.8cm);
    
    \draw (1.5 cm,1) node  {\scriptsize{$0$}};
    \draw (4.5 cm,-2.2) node  {\scriptsize{$2$}};
    \draw (3 cm,1) node  {\scriptsize{$0$}};
    \draw (4.5 cm,1) node  {\scriptsize{$0$}};
    \draw (6 cm,1) node  {\scriptsize{$0$}};
    \draw (7.5 cm,1) node  {\scriptsize{$0$}};
    
    \draw (0,-1) node (y)  {};
    \end{tikzpicture}
    & $x_4$&$e_{2,3,4,5,6}+e_{1,2,3,4^2,5}$
\\
\hline
(v) & $A_1^3$ &
\begin{tikzpicture}[scale=.3, baseline=(y)]
    \draw[thick] (1.5 cm,0) circle (.2cm);
    \draw[thick] (3 cm,0) circle (.2cm);
    \draw[thick] (4.5 cm,0) circle (.2cm);
    \draw[thick] (6 cm,0) circle (.2cm);
    \draw[thick] (7.5 cm,0) circle (.2cm);
    \draw[thick] (4.5 cm,-1.2) circle (.2cm);
    
    \draw[thick] (1.7 cm,0) -- +(1.1 cm,0);
    \draw[thick] (3.2 cm,0) -- +(1.1 cm,0);
    \draw[thick] (4.7 cm,0) -- +(1.1 cm,0);
    \draw[thick] (6.2 cm,0) -- +(1.1 cm,0);
    \draw[thick] (4.5 cm,-0.2) -- +(0,-0.8cm);
    
    \draw (1.5 cm,1) node  {\scriptsize{$0$}};
    \draw (4.5 cm,-2.2) node  {\scriptsize{$0$}};
    \draw (3 cm,1) node  {\scriptsize{$0$}};
    \draw (4.5 cm,1) node  {\scriptsize{$1$}};
    \draw (6 cm,1) node  {\scriptsize{$0$}};
    \draw (7.5 cm,1) node  {\scriptsize{$0$}};
    
    \draw (0,-1) node (y)  {};
    \end{tikzpicture}
    & $x_5$&$e_{1,2,3^2,4^2,5}+e_{1,2,3,4^2,5,6}+e_{2,3,4^2,5^2,6}$\\
\hline
(vi) & $A_2A_1$ &
\begin{tikzpicture}[scale=.3, baseline=(y)]
    \draw[thick] (1.5 cm,0) circle (.2cm);
    \draw[thick] (3 cm,0) circle (.2cm);
    \draw[thick] (4.5 cm,0) circle (.2cm);
    \draw[thick] (6 cm,0) circle (.2cm);
    \draw[thick] (7.5 cm,0) circle (.2cm);
    \draw[thick] (4.5 cm,-1.2) circle (.2cm);
    
    \draw[thick] (1.7 cm,0) -- +(1.1 cm,0);
    \draw[thick] (3.2 cm,0) -- +(1.1 cm,0);
    \draw[thick] (4.7 cm,0) -- +(1.1 cm,0);
    \draw[thick] (6.2 cm,0) -- +(1.1 cm,0);
    \draw[thick] (4.5 cm,-0.2) -- +(0,-0.8cm);
    
    \draw (1.5 cm,1) node  {\scriptsize{$1$}};
    \draw (4.5 cm,-2.2) node  {\scriptsize{$1$}};
    \draw (3 cm,1) node  {\scriptsize{$0$}};
    \draw (4.5 cm,1) node  {\scriptsize{$0$}};
    \draw (6 cm,1) node  {\scriptsize{$0$}};
    \draw (7.5 cm,1) node  {\scriptsize{$1$}};
    
    \draw (0,-1) node (y)  {};
    \end{tikzpicture}
    & $x_6$&$e_{1,3,4,5,6}+e_{2,3,4,5,6}+e_{1,2,3,4^2,5}$ \\
\hline
(vii) & $A_3$ &
\begin{tikzpicture}[scale=.3, baseline=(y)]
    \draw[thick] (1.5 cm,0) circle (.2cm);
    \draw[thick] (3 cm,0) circle (.2cm);
    \draw[thick] (4.5 cm,0) circle (.2cm);
    \draw[thick] (6 cm,0) circle (.2cm);
    \draw[thick] (7.5 cm,0) circle (.2cm);
    \draw[thick] (4.5 cm,-1.2) circle (.2cm);
    
    \draw[thick] (1.7 cm,0) -- +(1.1 cm,0);
    \draw[thick] (3.2 cm,0) -- +(1.1 cm,0);
    \draw[thick] (4.7 cm,0) -- +(1.1 cm,0);
    \draw[thick] (6.2 cm,0) -- +(1.1 cm,0);
    \draw[thick] (4.5 cm,-0.2) -- +(0,-0.8cm);
    
    \draw (1.5 cm,1) node  {\scriptsize{$1$}};
    \draw (4.5 cm,-2.2) node  {\scriptsize{$2$}};
    \draw (3 cm,1) node  {\scriptsize{$0$}};
    \draw (4.5 cm,1) node  {\scriptsize{$0$}};
    \draw (6 cm,1) node  {\scriptsize{$0$}};
    \draw (7.5 cm,1) node  {\scriptsize{$1$}};
    
    \draw (0,-1) node (y)  {};
    \end{tikzpicture}
    & $x_7$&$e_{2,3,4}+e_{2,4,5}+e_{1,3,4,5,6}$\\
\hline
(viii) & $A_2A_1^2$ &
\begin{tikzpicture}[scale=.3, baseline=(y)]
    \draw[thick] (1.5 cm,0) circle (.2cm);
    \draw[thick] (3 cm,0) circle (.2cm);
    \draw[thick] (4.5 cm,0) circle (.2cm);
    \draw[thick] (6 cm,0) circle (.2cm);
    \draw[thick] (7.5 cm,0) circle (.2cm);
    \draw[thick] (4.5 cm,-1.2) circle (.2cm);
    
    \draw[thick] (1.7 cm,0) -- +(1.1 cm,0);
    \draw[thick] (3.2 cm,0) -- +(1.1 cm,0);
    \draw[thick] (4.7 cm,0) -- +(1.1 cm,0);
    \draw[thick] (6.2 cm,0) -- +(1.1 cm,0);
    \draw[thick] (4.5 cm,-0.2) -- +(0,-0.8cm);
    
    \draw (1.5 cm,1) node  {\scriptsize{$0$}};
    \draw (4.5 cm,-2.2) node  {\scriptsize{$0$}};
    \draw (3 cm,1) node  {\scriptsize{$1$}};
    \draw (4.5 cm,1) node  {\scriptsize{$0$}};
    \draw (6 cm,1) node  {\scriptsize{$1$}};
    \draw (7.5 cm,1) node  {\scriptsize{$0$}};
    
    \draw (0,-1) node (y)  {};
    \end{tikzpicture}
    & $x_8$&$e_{1,2,3,4,5}+e_{1,3,4,5,6}+e_{2,3,4^2,5}+e_{2,3,4,5,6}$\\
\hline
(ix) & $A_2^2$ &
\begin{tikzpicture}[scale=.3, baseline=(y)]
    \draw[thick] (1.5 cm,0) circle (.2cm);
    \draw[thick] (3 cm,0) circle (.2cm);
    \draw[thick] (4.5 cm,0) circle (.2cm);
    \draw[thick] (6 cm,0) circle (.2cm);
    \draw[thick] (7.5 cm,0) circle (.2cm);
    \draw[thick] (4.5 cm,-1.2) circle (.2cm);
    
    \draw[thick] (1.7 cm,0) -- +(1.1 cm,0);
    \draw[thick] (3.2 cm,0) -- +(1.1 cm,0);
    \draw[thick] (4.7 cm,0) -- +(1.1 cm,0);
    \draw[thick] (6.2 cm,0) -- +(1.1 cm,0);
    \draw[thick] (4.5 cm,-0.2) -- +(0,-0.8cm);
    
    \draw (1.5 cm,1) node  {\scriptsize{$2$}};
    \draw (4.5 cm,-2.2) node  {\scriptsize{$0$}};
    \draw (3 cm,1) node  {\scriptsize{$0$}};
    \draw (4.5 cm,1) node  {\scriptsize{$0$}};
    \draw (6 cm,1) node  {\scriptsize{$0$}};
    \draw (7.5 cm,1) node  {\scriptsize{$2$}};
    
    \draw (0,-1) node (y)  {};
    \end{tikzpicture}
    & $x_9$&$e_{1,2,3,4}+e_{1,3,4,5}+e_{2,4,5,6}+e_{3,4,5,6}$\\
\hline
(x) & $A_3A_1$ &
\begin{tikzpicture}[scale=.3, baseline=(y)]
    \draw[thick] (1.5 cm,0) circle (.2cm);
    \draw[thick] (3 cm,0) circle (.2cm);
    \draw[thick] (4.5 cm,0) circle (.2cm);
    \draw[thick] (6 cm,0) circle (.2cm);
    \draw[thick] (7.5 cm,0) circle (.2cm);
    \draw[thick] (4.5 cm,-1.2) circle (.2cm);
    
    \draw[thick] (1.7 cm,0) -- +(1.1 cm,0);
    \draw[thick] (3.2 cm,0) -- +(1.1 cm,0);
    \draw[thick] (4.7 cm,0) -- +(1.1 cm,0);
    \draw[thick] (6.2 cm,0) -- +(1.1 cm,0);
    \draw[thick] (4.5 cm,-0.2) -- +(0,-0.8cm);
    
    \draw (1.5 cm,1) node  {\scriptsize{$0$}};
    \draw (4.5 cm,-2.2) node  {\scriptsize{$1$}};
    \draw (3 cm,1) node  {\scriptsize{$1$}};
    \draw (4.5 cm,1) node  {\scriptsize{$0$}};
    \draw (6 cm,1) node  {\scriptsize{$1$}};
    \draw (7.5 cm,1) node  {\scriptsize{$0$}};
    
    \draw (0,-1) node (y)  {};
    \end{tikzpicture}
    & $x_{10}$&$e_{3,4,5}+e_{1,2,3,4}+e_{2,4,5,6}+e_{1,3,4,5,6}$\\
\hline
(xi) & $A_4$ &
\begin{tikzpicture}[scale=.3, baseline=(y)]
    \draw[thick] (1.5 cm,0) circle (.2cm);
    \draw[thick] (3 cm,0) circle (.2cm);
    \draw[thick] (4.5 cm,0) circle (.2cm);
    \draw[thick] (6 cm,0) circle (.2cm);
    \draw[thick] (7.5 cm,0) circle (.2cm);
    \draw[thick] (4.5 cm,-1.2) circle (.2cm);
    
    \draw[thick] (1.7 cm,0) -- +(1.1 cm,0);
    \draw[thick] (3.2 cm,0) -- +(1.1 cm,0);
    \draw[thick] (4.7 cm,0) -- +(1.1 cm,0);
    \draw[thick] (6.2 cm,0) -- +(1.1 cm,0);
    \draw[thick] (4.5 cm,-0.2) -- +(0,-0.8cm);
    
    \draw (1.5 cm,1) node  {\scriptsize{$2$}};
    \draw (4.5 cm,-2.2) node  {\scriptsize{$2$}};
    \draw (3 cm,1) node  {\scriptsize{$0$}};
    \draw (4.5 cm,1) node  {\scriptsize{$0$}};
    \draw (6 cm,1) node  {\scriptsize{$0$}};
    \draw (7.5 cm,1) node  {\scriptsize{$2$}};
    
    \draw (0,-1) node (y)  {};
    \end{tikzpicture}
    & $x_{11}$&$e_{5,6}+e_{1,3,4}+e_{2,3,4}+e_{2,4,5}$\\
\hline
(xii) & $D_4$ &
\begin{tikzpicture}[scale=.3, baseline=(y)]
    \draw[thick] (1.5 cm,0) circle (.2cm);
    \draw[thick] (3 cm,0) circle (.2cm);
    \draw[thick] (4.5 cm,0) circle (.2cm);
    \draw[thick] (6 cm,0) circle (.2cm);
    \draw[thick] (7.5 cm,0) circle (.2cm);
    \draw[thick] (4.5 cm,-1.2) circle (.2cm);
    
    \draw[thick] (1.7 cm,0) -- +(1.1 cm,0);
    \draw[thick] (3.2 cm,0) -- +(1.1 cm,0);
    \draw[thick] (4.7 cm,0) -- +(1.1 cm,0);
    \draw[thick] (6.2 cm,0) -- +(1.1 cm,0);
    \draw[thick] (4.5 cm,-0.2) -- +(0,-0.8cm);
    
    \draw (1.5 cm,1) node  {\scriptsize{$0$}};
    \draw (4.5 cm,-2.2) node  {\scriptsize{$2$}};
    \draw (3 cm,1) node  {\scriptsize{$0$}};
    \draw (4.5 cm,1) node  {\scriptsize{$2$}};
    \draw (6 cm,1) node  {\scriptsize{$0$}};
    \draw (7.5 cm,1) node  {\scriptsize{$0$}};
    
    \draw (0,-1) node (y)  {};
    \end{tikzpicture}
    & $x_{12}$&$e_2+e_{1,3,4}+e_{3,4,5}+e_{4,5,6}$\\
\hline
(xiii) & $D_4(a_1)$ &
\begin{tikzpicture}[scale=.3, baseline=(y)]
    \draw[thick] (1.5 cm,0) circle (.2cm);
    \draw[thick] (3 cm,0) circle (.2cm);
    \draw[thick] (4.5 cm,0) circle (.2cm);
    \draw[thick] (6 cm,0) circle (.2cm);
    \draw[thick] (7.5 cm,0) circle (.2cm);
    \draw[thick] (4.5 cm,-1.2) circle (.2cm);
    
    \draw[thick] (1.7 cm,0) -- +(1.1 cm,0);
    \draw[thick] (3.2 cm,0) -- +(1.1 cm,0);
    \draw[thick] (4.7 cm,0) -- +(1.1 cm,0);
    \draw[thick] (6.2 cm,0) -- +(1.1 cm,0);
    \draw[thick] (4.5 cm,-0.2) -- +(0,-0.8cm);
    
    \draw (1.5 cm,1) node  {\scriptsize{$0$}};
    \draw (4.5 cm,-2.2) node  {\scriptsize{$0$}};
    \draw (3 cm,1) node  {\scriptsize{$0$}};
    \draw (4.5 cm,1) node  {\scriptsize{$2$}};
    \draw (6 cm,1) node  {\scriptsize{$0$}};
    \draw (7.5 cm,1) node  {\scriptsize{$0$}};
    
    \draw (0,-1) node (y)  {};
    \end{tikzpicture}
    & $x_{13}$&$e_{3,4,5}+e_{4,5,6}+e_{1,2,3,4}+e_{2,4,5,6}$\\
\hline
(xiv) & $A_2^2A_1$ &
\begin{tikzpicture}[scale=.3, baseline=(y)]
    \draw[thick] (1.5 cm,0) circle (.2cm);
    \draw[thick] (3 cm,0) circle (.2cm);
    \draw[thick] (4.5 cm,0) circle (.2cm);
    \draw[thick] (6 cm,0) circle (.2cm);
    \draw[thick] (7.5 cm,0) circle (.2cm);
    \draw[thick] (4.5 cm,-1.2) circle (.2cm);
    
    \draw[thick] (1.7 cm,0) -- +(1.1 cm,0);
    \draw[thick] (3.2 cm,0) -- +(1.1 cm,0);
    \draw[thick] (4.7 cm,0) -- +(1.1 cm,0);
    \draw[thick] (6.2 cm,0) -- +(1.1 cm,0);
    \draw[thick] (4.5 cm,-0.2) -- +(0,-0.8cm);
    
    \draw (1.5 cm,1) node  {\scriptsize{$1$}};
    \draw (4.5 cm,-2.2) node  {\scriptsize{$0$}};
    \draw (3 cm,1) node  {\scriptsize{$0$}};
    \draw (4.5 cm,1) node  {\scriptsize{$1$}};
    \draw (6 cm,1) node  {\scriptsize{$0$}};
    \draw (7.5 cm,1) node  {\scriptsize{$1$}};
    
    \draw (0,-1) node (y)  {};
    \end{tikzpicture}
    & $x_{14}$&$e_{1,2,3,4}+e_{1,3,4,5}+e_{2,4,5,6}+e_{3,4,5,6}+e_{2,3,4^2,5}$ \\
\hline
(xv) & $A_4A_1$ &
\begin{tikzpicture}[scale=.3, baseline=(y)]
    \draw[thick] (1.5 cm,0) circle (.2cm);
    \draw[thick] (3 cm,0) circle (.2cm);
    \draw[thick] (4.5 cm,0) circle (.2cm);
    \draw[thick] (6 cm,0) circle (.2cm);
    \draw[thick] (7.5 cm,0) circle (.2cm);
    \draw[thick] (4.5 cm,-1.2) circle (.2cm);
    
    \draw[thick] (1.7 cm,0) -- +(1.1 cm,0);
    \draw[thick] (3.2 cm,0) -- +(1.1 cm,0);
    \draw[thick] (4.7 cm,0) -- +(1.1 cm,0);
    \draw[thick] (6.2 cm,0) -- +(1.1 cm,0);
    \draw[thick] (4.5 cm,-0.2) -- +(0,-0.8cm);
    
    \draw (1.5 cm,1) node  {\scriptsize{$1$}};
    \draw (4.5 cm,-2.2) node  {\scriptsize{$1$}};
    \draw (3 cm,1) node  {\scriptsize{$1$}};
    \draw (4.5 cm,1) node  {\scriptsize{$0$}};
    \draw (6 cm,1) node  {\scriptsize{$1$}};
    \draw (7.5 cm,1) node  {\scriptsize{$1$}};
    
    \draw (0,-1) node (y)  {};
    \end{tikzpicture}
    & $x_{15}$&$e_{5,6}+e_{1,3,4}+e_{2,3,4}+e_{2,4,5}+e_{3,4,5}$\\
\hline
(xvi) & $A_5$ &
\begin{tikzpicture}[scale=.3, baseline=(y)]
    \draw[thick] (1.5 cm,0) circle (.2cm);
    \draw[thick] (3 cm,0) circle (.2cm);
    \draw[thick] (4.5 cm,0) circle (.2cm);
    \draw[thick] (6 cm,0) circle (.2cm);
    \draw[thick] (7.5 cm,0) circle (.2cm);
    \draw[thick] (4.5 cm,-1.2) circle (.2cm);
    
    \draw[thick] (1.7 cm,0) -- +(1.1 cm,0);
    \draw[thick] (3.2 cm,0) -- +(1.1 cm,0);
    \draw[thick] (4.7 cm,0) -- +(1.1 cm,0);
    \draw[thick] (6.2 cm,0) -- +(1.1 cm,0);
    \draw[thick] (4.5 cm,-0.2) -- +(0,-0.8cm);
    
    \draw (1.5 cm,1) node  {\scriptsize{$2$}};
    \draw (4.5 cm,-2.2) node  {\scriptsize{$1$}};
    \draw (3 cm,1) node  {\scriptsize{$1$}};
    \draw (4.5 cm,1) node  {\scriptsize{$0$}};
    \draw (6 cm,1) node  {\scriptsize{$1$}};
    \draw (7.5 cm,1) node  {\scriptsize{$2$}};
    
    \draw (0,-1) node (y)  {};
    \end{tikzpicture}
    & $x_{16}$&$e_1+e_6+e_{2,3,4}+e_{2,4,5}+e_{3,4,5}$ \\
\hline
(xvii) & $D_5$ &
\begin{tikzpicture}[scale=.3, baseline=(y)]
    \draw[thick] (1.5 cm,0) circle (.2cm);
    \draw[thick] (3 cm,0) circle (.2cm);
    \draw[thick] (4.5 cm,0) circle (.2cm);
    \draw[thick] (6 cm,0) circle (.2cm);
    \draw[thick] (7.5 cm,0) circle (.2cm);
    \draw[thick] (4.5 cm,-1.2) circle (.2cm);
    
    \draw[thick] (1.7 cm,0) -- +(1.1 cm,0);
    \draw[thick] (3.2 cm,0) -- +(1.1 cm,0);
    \draw[thick] (4.7 cm,0) -- +(1.1 cm,0);
    \draw[thick] (6.2 cm,0) -- +(1.1 cm,0);
    \draw[thick] (4.5 cm,-0.2) -- +(0,-0.8cm);
    
    \draw (1.5 cm,1) node  {\scriptsize{$2$}};
    \draw (4.5 cm,-2.2) node  {\scriptsize{$2$}};
    \draw (3 cm,1) node  {\scriptsize{$0$}};
    \draw (4.5 cm,1) node  {\scriptsize{$2$}};
    \draw (6 cm,1) node  {\scriptsize{$0$}};
    \draw (7.5 cm,1) node  {\scriptsize{$2$}};
    
    \draw (0,-1) node (y)  {};
    \end{tikzpicture}
    & $x_{17}$&$e_2+e_6+e_{1,3}+e_{3,4}+e_{4,5}$\\
\hline
(xviii) & $D_5(a_1)$ &
\begin{tikzpicture}[scale=.3, baseline=(y)]
    \draw[thick] (1.5 cm,0) circle (.2cm);
    \draw[thick] (3 cm,0) circle (.2cm);
    \draw[thick] (4.5 cm,0) circle (.2cm);
    \draw[thick] (6 cm,0) circle (.2cm);
    \draw[thick] (7.5 cm,0) circle (.2cm);
    \draw[thick] (4.5 cm,-1.2) circle (.2cm);
    
    \draw[thick] (1.7 cm,0) -- +(1.1 cm,0);
    \draw[thick] (3.2 cm,0) -- +(1.1 cm,0);
    \draw[thick] (4.7 cm,0) -- +(1.1 cm,0);
    \draw[thick] (6.2 cm,0) -- +(1.1 cm,0);
    \draw[thick] (4.5 cm,-0.2) -- +(0,-0.8cm);
    
    \draw (1.5 cm,1) node  {\scriptsize{$1$}};
    \draw (4.5 cm,-2.2) node  {\scriptsize{$2$}};
    \draw (3 cm,1) node  {\scriptsize{$1$}};
    \draw (4.5 cm,1) node  {\scriptsize{$0$}};
    \draw (6 cm,1) node  {\scriptsize{$1$}};
    \draw (7.5 cm,1) node  {\scriptsize{$1$}};
    
    \draw (0,-1) node (y)  {};
    \end{tikzpicture}
    & $x_{18}$&$e_{1,3}+e_{2,4}+e_{5,6}+e_{3,4,5}+e_{4,5,6}$\\
\hline
(ixx) & $E_6(a_1)$ &
\begin{tikzpicture}[scale=.3, baseline=(y)]
    \draw[thick] (1.5 cm,0) circle (.2cm);
    \draw[thick] (3 cm,0) circle (.2cm);
    \draw[thick] (4.5 cm,0) circle (.2cm);
    \draw[thick] (6 cm,0) circle (.2cm);
    \draw[thick] (7.5 cm,0) circle (.2cm);
    \draw[thick] (4.5 cm,-1.2) circle (.2cm);
    
    \draw[thick] (1.7 cm,0) -- +(1.1 cm,0);
    \draw[thick] (3.2 cm,0) -- +(1.1 cm,0);
    \draw[thick] (4.7 cm,0) -- +(1.1 cm,0);
    \draw[thick] (6.2 cm,0) -- +(1.1 cm,0);
    \draw[thick] (4.5 cm,-0.2) -- +(0,-0.8cm);
    
    \draw (1.5 cm,1) node  {\scriptsize{$2$}};
    \draw (4.5 cm,-2.2) node  {\scriptsize{$2$}};
    \draw (3 cm,1) node  {\scriptsize{$2$}};
    \draw (4.5 cm,1) node  {\scriptsize{$0$}};
    \draw (6 cm,1) node  {\scriptsize{$2$}};
    \draw (7.5 cm,1) node  {\scriptsize{$2$}};
    
    \draw (0,-1) node (y)  {};
    \end{tikzpicture}
    & $x_{19}$&$e_1+e_3+e_{2,4}+e_{4,5}+e_5+e_6$\\
\hline
(xx) & $E_6(a_3)$ &
\begin{tikzpicture}[scale=.3, baseline=(y)]
    \draw[thick] (1.5 cm,0) circle (.2cm);
    \draw[thick] (3 cm,0) circle (.2cm);
    \draw[thick] (4.5 cm,0) circle (.2cm);
    \draw[thick] (6 cm,0) circle (.2cm);
    \draw[thick] (7.5 cm,0) circle (.2cm);
    \draw[thick] (4.5 cm,-1.2) circle (.2cm);
    
    \draw[thick] (1.7 cm,0) -- +(1.1 cm,0);
    \draw[thick] (3.2 cm,0) -- +(1.1 cm,0);
    \draw[thick] (4.7 cm,0) -- +(1.1 cm,0);
    \draw[thick] (6.2 cm,0) -- +(1.1 cm,0);
    \draw[thick] (4.5 cm,-0.2) -- +(0,-0.8cm);
    
    \draw (1.5 cm,1) node  {\scriptsize{$2$}};
    \draw (4.5 cm,-2.2) node  {\scriptsize{$0$}};
    \draw (3 cm,1) node  {\scriptsize{$0$}};
    \draw (4.5 cm,1) node  {\scriptsize{$2$}};
    \draw (6 cm,1) node  {\scriptsize{$0$}};
    \draw (7.5 cm,1) node  {\scriptsize{$2$}};
    
    \draw (0,-1) node (y)  {};
    \end{tikzpicture}
    & $x_{20}$&$e_1+e_{3,4}+e_{2,4}+e_{2,4,5}+e_{2,3,4,5}+e_{5,6}$\\
\hline
(xxi) & $E_6$ &
\begin{tikzpicture}[scale=.3, baseline=(y)]
    \draw[thick] (1.5 cm,0) circle (.2cm);
    \draw[thick] (3 cm,0) circle (.2cm);
    \draw[thick] (4.5 cm,0) circle (.2cm);
    \draw[thick] (6 cm,0) circle (.2cm);
    \draw[thick] (7.5 cm,0) circle (.2cm);
    \draw[thick] (4.5 cm,-1.2) circle (.2cm);
    
    \draw[thick] (1.7 cm,0) -- +(1.1 cm,0);
    \draw[thick] (3.2 cm,0) -- +(1.1 cm,0);
    \draw[thick] (4.7 cm,0) -- +(1.1 cm,0);
    \draw[thick] (6.2 cm,0) -- +(1.1 cm,0);
    \draw[thick] (4.5 cm,-0.2) -- +(0,-0.8cm);
    
    \draw (1.5 cm,1) node  {\scriptsize{$2$}};
    \draw (4.5 cm,-2.2) node  {\scriptsize{$2$}};
    \draw (3 cm,1) node  {\scriptsize{$2$}};
    \draw (4.5 cm,1) node  {\scriptsize{$2$}};
    \draw (6 cm,1) node  {\scriptsize{$2$}};
    \draw (7.5 cm,1) node  {\scriptsize{$2$}};
    
    \draw (0,-1) node (y)  {};
    \end{tikzpicture}
    & $x_{21}$&$e_1+e_2+e_3+e_4+e_5+e_6$\\
\hline
\end{longtable}
\end{center}
Applying the programme results in the same pieces as in characteristic 0, i.e. we have $\cN_i=O_i$ for all $i=1,\ldots, 21$ with the same notation as above.
\begin{thm}[Nilpotent pieces for $E_6$ in characteristic $2$ and $3$]
We use the same notation as above. Then the pieces for $E_6$ in both characteristic $2$ and $3$ are given by 
$\cN_i=\cO_i$ for all $1\leq i\leq 21$. In particular, the nilpotent pieces form a partition of the nilpotent variety and therefore agree with the CP-pieces by Remark \ref{rem_CP-pieces}.
\end{thm}
\subsection*{Acknowledgemnts}
I would like to thank my PhD advisor Gunter Malle for his suggestions and comments on earlier versions.\\
This work was financially supported by the SFB-TRR 195 of the German Research Foundation (DFG) in the scope of project A2 (	Generalised Gelfand-Graev representations, unipotent classes and nilpotent orbits).


\begin{thebibliography}{131}
\bibitem{Borel}
{\sc A. Borel}, {\it Linear algebraic groups} (Second enlarged edition), Springer-Verlag, New York, 1991.
\bibitem{Magma}
{\sc Wieb Bosma, John Cannon, and Catherine Playoust}, The Magma algebra system. I. The user language, J. Symbolic Comput., 24 (1997), 235--265.
\bibitem{MagmaHandbook}
{\sc W. Bosma, J. J. Cannon, C. Fieker, A. Steel (eds.)}, Handbook of Magma functions, Version 2.26-4 (2021), 6347 pages.
\bibitem{Bourbaki}
{\sc N. Bourbaki},{\it Groupes et algèbres de Lie, Chapitres 4, 5 et 6}, Masson, Paris, 1981.
\bibitem{Carter72}
{\sc R. W. Carter}, {\it Simple groups of Lie type}, Wiley, New York, 1972; reprinted 1989 as Wiley Classics Library Edition.
\bibitem{Carter85} 
{\sc R. W.~Carter}, {\it Finite groups of Lie type: Conjugacy classes 
and complex characters}, Wiley, New York, 1985.

\bibitem{ClarkePremet}
{\sc M. C. Clarke and A. Premet}, The Hesselink stratification of nullcones
and base change, Invent. math. {\bf 191} (2013), 631--669.
\bibitem{Geck}
{\sc M. Geck}, On the construction of semisimple Lie algebras and Chevalley groups, Proc. Amer. Math. Soc. {\bf 145}:8 (2017), 3233–3247.
\bibitem{Geck20}
{\sc M. Geck}, Generalised Gelfand-Graev representations in bad characteristic? Transform. Groups {\bf 26} (2021), 305--326.
\bibitem{Hesselink}
{\sc W. Hesselink}, The nullcone of the Lie algebra of $G_2$, Indag. Math. {\bf 30}:4 (2019), 623–648.
\bibitem{Humphreys75}
{\sc J. E. Humphreys}, {\it Linear algebraic groups}, Springer-Verlag, New York, 1975.
\bibitem{Humphreys80}
{\sc J. E. Humphreys}, {\it Introduction to Lie algebras and representation theory}. Grad-
uate Texts in Mathematics, 9. Springer-Verlag, New York, Second printing, 1980.
\bibitem{Lang}
{\sc S. Lang}, {\it Algebra} (2nd edition), Addison-Wesley Publishing Company, 1984.
\bibitem{LiebeckSeitz}
{\sc M. W. Liebeck and G. M. Seitz}, {\it Unipotent and nilpotent classes in simple algebraic groups and Lie algebras}, American Mathematical Society, 2012.
\bibitem{Lusztig1}
{\sc G. Lusztig}, Unipotent elements in small characteristic, 
Transform. Groups {\bf 10} (2005), no. 3--4, 449--487.
\bibitem{Lusztig2}
{\sc G. Lusztig}, Unipotent elements in small characteristic II, 
Transform. Groups {\bf 13} (2008), no. 3--4, 773--797.
\bibitem{Lusztig3}
{\sc G. Lusztig}, Unipotent elements in small characteristic III, 
J. Algebra {\bf 329} (2011), 163--189.
\bibitem{Lusztig4}
{\sc G. Lusztig}, Unipotent elements in small characteristic IV, 
Transform. Groups {\bf 15} (2010), no. 4, 921--936.
\bibitem{MalleTesterman}
{\sc G. Malle and D. Testerman}, {\it Linear algebraic groups and finite groups of Lie type}, Cambridge University Press, 2011.
\bibitem{Premet}
{\sc  A. Premet}, Nilpotent orbits in good characteristic and the Kempf–Rousseau theory (Special issue celebrating the 80th birthday of Robert Steinberg). J. Algebra {\bf 260} (2003), 338–366.
\bibitem{Pommerening1}
{\sc K. Pommerening}, Über die unipotenten Klassen reduktiver Gruppen. J. Algebra {\bf 49} (1977), 525--536.
\bibitem{Pommerening2}
{\sc K. Pommerening}, Über die unipotenten Klassen reduktiver Gruppen II. J. Algebra {\bf 65} (1980), 373--398.
\bibitem{Spaltenstein}
{\sc N. Spaltenstein}, Nilpotent classes in Lie algebras of type $F_4$ over fields of characteristic 2, J. Fac. Sci. Univ. Tokyo. Sect. 1 A, Mathematics, {\bf 30} (1984), 517--524.
\bibitem{Springer}
{\sc T. A. Springer}, {\it Linear algebraic groups.} Second edition. Progress in Mathematics, 9. Birkhäuser, Boston, 1998.
\bibitem{Stuhler}
{\sc U. Stuhler}, Unipotente und nilpotente Klassen in einfachen Gruppen und Liealgebren vom Typ $G_2$,
Indag. Math. (Proceedings) {\bf 74} (1971), 365--378.
\bibitem{Xue14}
{\sc T. Xue}, Nilpotent coadjoint orbits in small characteristic.
J. Algebra {\bf 397} (2014), 111--140.
\bibitem{Xue17}
{\sc T. Xue}, Springer correspondence for exceptional Lie algebras and their duals in small characteristic.
J. Lie Theory {\bf 27} (2017), no.2, 357--375.
\end{thebibliography}
\end{document}